\documentclass{aims}
\usepackage{amsmath}
  \usepackage{paralist}
  \usepackage{graphics} %% add this and next lines if pictures should be in esp format
  \usepackage{epsfig} %For pictures: screened artwork should be set up with an 85 or 100 line screen
\usepackage{graphicx}  \usepackage{epstopdf}%This is to transfer .eps figure to .pdf figure; please compile your paper using PDFLeTex or PDFTeXify.
 \usepackage[colorlinks=true]{hyperref}
\hypersetup{urlcolor=blue, citecolor=red}

\def\currentvolume{X}
 \def\currentissue{X}
  \def\currentyear{200X}
   \def\currentmonth{XX}
    \def\ppages{X--XX}
     \def\DOI{10.3934/xx.xx.xx.xx}

\newcommand{\N}{\mathbb{N}}
\newcommand{\R}{\mathbb{R}}
\newcommand{\C}{\mathbb{C}}
\newcommand{\A}{\mathcal{A}}
\newcommand{\B}{\mathcal{B}}
\newcommand{\X}{\mathcal{X}}
\newcommand{\Y}{\mathcal{Y}}
\renewcommand{\L}{\mathcal{L}}
\newcommand{\supp}{\textnormal{supp }}
\newcommand{\sign}{\text{sign}}
\newcommand{\ess}{\mathnormal{ess}}

\newtheorem{theorem}{Theorem}[section]

\newtheorem{lemma}[theorem]{Lemma}

\theoremstyle{definition}
\newtheorem{definition}[theorem]{Definition}
\newtheorem{remark}{Remark}
\newtheorem{assumption}{Assumption}

\title[A two-phase population model] %Use the shortened version of the full title
      {Spectral theory and time asymptotics of size-structured two-phase population models}

\author[Mustapha Mokhtar-Kharroubi and Quentin Richard]{}


 \email{mustapha.mokhtar-kharroubi@univ-fcomte.fr}
 \email{quentin.richard@math.cnrs.fr}

\begin{document}
\maketitle

% Enter the first author's name and address:
\centerline{\scshape Mustapha Mokhtar-Kharroubi}
\medskip
{\footnotesize
% please put the address of the first author
 \centerline{UMR 6623 Laboratoire de Mathématiques de Besan\c con}
   \centerline{Universit\'e Bourgogne Franche-Comt\'e}
   \centerline{Besan\c con, 25000, FRANCE}
} % Do not forget to end the {\footnotesize by the sign }

\medskip

\centerline{\scshape Quentin Richard}
\medskip
{\footnotesize
 % please put the address of the second  and third author
 \centerline{UMR 5251 Institut de Math\'ematiques de Bordeaux}
   \centerline{Universit\'e de Bordeaux}
   \centerline{Talence, 33400, FRANCE}
}

\bigskip

% The name of the associate editor will be entered by an editorial staff
% "Communicated by the associate editor name" is not needed for special issue.
 \centerline{(Communicated by the associate editor name)}

%The abstract of your paper
\begin{abstract}
This work provides a general spectral analysis of
size-structured two-phase population models. Systematic functional analytic results are given. We deal first with the case of finite maximal size. We characterize the irreducibility of the corresponding $L^{1}$ semigroup in terms of properties of the different parameters of the system. We characterize also the spectral gap property of the semigroup. It turns out
that the irreducibility of the semigroup implies the existence of the spectral gap. In particular, we provide a general criterion for asynchronous
exponential growth. We show also how to deal with time asymptotics in case of lack of irreducibility. Finally, we extend the theory to the case of
infinite maximal size.
\end{abstract}

%The title of your section 1
\section{Introduction}

Time asymptotics of structured biological populations are widely discussed in the literature on population dynamics (see e.g. \cite{Cushing98, Inaba2017, Magal2008}). When describing the evolution of cell populations, one can consider that individuals may be proliferating or quiescent, i.e. in two different stages in their life called `active' and `resting'. Taking into account maturity as a structure variable, Rotenberg \cite{Rotenberg83} introduced in this context the first structured population model (see also the paper of Dyson,  Villella-Bressan and Webb \cite{DysonWebb99}). Since the size plays an important role in the dynamics of cells, Gyllenberg and Webb introduced \cite{GyllenbergWebb87} the first size and age-structured population model with a quiescence state. They prove under general hypotheses the asychronous exponential growth behavior of the population. We note that size-structured population model appeared in a work by Sinko and Streifer \cite{Sinko67} (see e.g. \cite{Webb2008} for more size-structured models). Among the age-structured models in this context, we can look at the works of Arino, S\'anchez and Webb \cite{ArinoWebb97} as well as Dyson, Villella-Bressan and Webb \cite{DysonWebb2002}. The same asymptotic behavior is proved for these models under general assumptions. Thereafter, Farkas and Hinow \cite{Farkas2010} introduced  a size-structured model. In a specific case for the reproduction function (more precisely in the case of equal mitosis, where the offspring is composed of two new daughter cells with size being half of the mother cell size), we can mention the works of Gyllenberg and Webb \cite{GyllenbergWebb90, GyllenbergWebb91}, Rossa \cite{Rossa97} as well as Bai and Cui \cite{BaiCui2010}. As far as the literature is concerned, the size-structured model of equal division without quiescence was investigated through spectral analysis of semigroups by Diekmann, Heijmans and Thieme \cite{DiekmannThieme84} and Greiner and Nagel \cite{GreinerNagel85} in the case of a finite maximal size $(m<\infty)$. More recently, the case of an infinite maximal size $(m=\infty)$ was treated by Mischler and Scher \cite{Mischler2016} and Bernard and Gabriel \cite{BernardGabriel2019}.

The goal of the present work is to provide a systematic spectral analysis of
the coupled linear structured population model considered by Farkas and
Hinow \cite{Farkas2010} 
\begin{equation}
\left\{ 
\begin{array}{rcl}
\partial _{t}u_{1}(t,s)+\partial _{s}(\gamma _{1}(s)u_{1}(t,s)) & = & -\mu
(s)u_{1}(t,s)+\int_{0}^{m}\beta (s,y)u_{1}(t,y)dy \\ 
&  & -c_{1}(s)u_{1}(t,s)+c_{2}(s)u_{2}(t,s), \\ 
\partial _{t}u_{2}(t,s)+\partial _{s}(\gamma _{2}(s)u_{2}(t,s)) & = & 
c_{1}(s)u_{1}(t,s)-c_{2}(s)u_{2}(t,s),
\end{array}
\right.  \label{Eq:Model}
\end{equation}
with Dirichlet boundary conditions 
\begin{equation}
u_{1}(t,0)=0,\qquad u_{2}(t,0)=0,\qquad \forall t\geq 0.
\label{Eq:BoundCond}
\end{equation}
The density of individuals in the active (resp. resting) stage of size $s\in
\lbrack 0,m]$ at time $t$ is denoted by $u_{1}(s,t)$ (resp. $u_{2}(s,t)$)
and 
\begin{equation*}
m<\infty
\end{equation*}
is the maximal size that can be reached. For each stage, the individuals
will grow respectively with the rate $\gamma _{1}$ and $\gamma _{2}$.
Furthermore, only proliferating individuals have a mortality rate denoted by 
$\mu $ and also can reproduce via the non-local integral recruitment term in %
\eqref{Eq:Model}. More precisely, $\beta (s,y)$ gives the rate at which an
individual of size $y$ produces offspring of size $s$. Finally, the
transition between the two life-stages is described by the size-dependent
functions $c_{1}$ and $c_{2}$.

In this paper, we deal also with the case of infinite maximal sizes
\begin{equation*}
m=\infty .
\end{equation*}
The natural functional space for such a system is
\begin{equation*}
\mathcal{X}:=L^{1}(0,m)\times L^{1}(0,m).
\end{equation*}
Our approach of \textit{asynchronous exponential growth} (see Definition \ref{Def:AEG}) of such a system is in the spirit of our previous work \cite{MokhtarRichard18}. The analysis relies on two mathematical ingredients:

(i) Check that the positive $C_{0}$-semigroup $\{T_{\A}(t)\}_{t\geq 0}$ which
governs this system has a \textit{spectral gap}, i.e.  
\begin{equation*}
\omega _{ess}<\omega_0 
\end{equation*}%
where $\omega_0 $ and $\omega _{ess}$ are respectively the type and the
essential type of $\{T_{\A}(t)\}_{t\geq 0}$ (see Definition \ref{Def:Ess}). (Note that $\omega_0 $ coincides with the spectral bound
$$s(\mathcal{A}):=\sup\{\Re(\lambda), \lambda\in \sigma(\A)\}$$
of its generator $\mathcal{A}$). 

(ii) Check that the $C_{0}$-semigroup $\{T_{\A}(t)\}_{t\geq 0}$ is \textit{irreducible} (see Definition \ref{Def:Irr}). To avoid any ambiguity, we will denote by $\omega_0(A)$ and $\omega_{\ess}(A)$ the type and essential type of $\{T_{\A}(t)\}_{t\geq 0}$ where $A$ is its generator.

Indeed, it is well known (see Theorem \ref{Thm:StandardAsync} below) that in a Banach lattice, a positive irreducible semigroup having a spectral gap, has an asynchronous exponential growth. Thus, the main goal of the paper is to check that the $C_0$-semigroup $\{T_{\A}(t)\}_{t\geq 0}$ is irreducible and has a spectral gap. Both issues are non trivial. Indeed, the irreducibility of $\{T_{\A}(t)\}_{t\geq 0}$ need not be satisfied in general because of the absence of an integral recruitment term in the second equation of \eqref{Eq:Model}. The spectral gap property is also a non trivial issue related to stability of essential type of perturbed semigroups.

Our general strategy consists in considering the semigroup governing \eqref{Eq:Model} as a perturbation of the one governing 
\begin{equation*}
\left\{ 
\begin{array}{rcl}
\partial _{t}u_{1}(t,s)+\partial _{s}(\gamma _{1}(s)u_{1}(t,s)) & = & -\mu
(s)u_{1}(t,s)-c_{1}(s)u_{1}(t,s)+c_{2}(s)u_{2}(t,s), \\ 
\partial _{t}u_{2}(t,s)+\partial _{s}(\gamma _{2}(s)u_{2}(t,s)) & = & 
c_{1}(s)u_{1}(t,s)-c_{2}(s)u_{2}(t,s),
\end{array}
\right.  
\end{equation*}
where the perturbation is given by
$$\X\ni \begin{pmatrix}
u_1 \\
u_2
\end{pmatrix}\longmapsto \begin{pmatrix}
\int_0^m \beta(s,y)u_1(y)dy \\
0
\end{pmatrix}\in \X.$$
Our assumptions are weaker than those given by Farkas and Hinow \cite{Farkas2010} and our construction is more systematic. We provide several new contributions. The most important ones are the following:

1. In all the paper, we will suppose that the operator
$$K:L^1(0,m)\ni u\mapsto \int_0^m \beta(\cdot, y)u(y)dy\in L^1(0,m)$$
is bounded; see Remark \ref{Rem:Bounded}. Moreover, to obtain the existence of a spectral gap we will suppose that $K$ is weakly compact, i.e. $K$ maps a bounded subset of $L^1(0,m)$ into a weakly compact one; see Remark \ref{Rem:Weak-Compact} for details on weak compactness. (This latter property covers e.g. the case where $\beta$ is continuous on $[0,m]\times [0,m]$, as assumed in \cite{Farkas2010}). However, one may note that this assumption precludes the case of equal mitosis, since it  corresponds to a Dirac mass ``kernel''
$$\beta(s,y)=\delta_{\{s=y/2\}}$$
whose corresponding operator is certainly not weakly compact.

2. We show that the three conditions 
\begin{equation}
\forall \varepsilon \in (0,m),\quad \int_{0}^{\varepsilon }\int_{\varepsilon
}^{m}\beta (s,y)dyds>0,  \label{Condition 1}
\end{equation}
\begin{equation}
\inf \supp c_{1}=0,  \label{Condition 2}
\end{equation}
\begin{equation}
\sup \supp c_{2}=m  \label{Condition 3}
\end{equation}
\textit{characterize} the irreducibility of $\{T_{\A}(t)\}_{t\geq 0}$, (see
Theorem \ref{Thm:Irr}) where $\inf \supp c_{1}$ is the infimum of the support of $
c_{1}$ and $\sup \supp c_{2}$ is the supremum of the support of $c_{2}$. (In particular, our result covers the sufficient condition given in \cite{Farkas2010}).

3. We show that the spectrum $\sigma (\mathcal{A})$ of the generator $%
\mathcal{A}$ of $\{T_{\A}(t)\}_{t\geq 0}$ is not empty, or equivalently 
\begin{equation*}
s(\mathcal{A})>-\infty 
\end{equation*}
($s(\mathcal{A})$ is the spectral bound of $\mathcal{A}$), \textit{if and
only if} 
\begin{equation}
\exists \ \varepsilon \in (0,m),\quad \int_{0}^{\varepsilon }\int_{\varepsilon
}^{m}\beta (s,y)dyds>0  \label{Condition 4}
\end{equation}
(see Theorem \ref{Thm:SpectBound}) and moreover, this \textit{characterizes} the property that $\{T_{\A}(t)\}_{t\geq
0}$ has a spectral gap (see Theorem \ref{Thm:Asynch}). (In particular, we extend the results given in \cite{Farkas2010}). Note that here
the irreducibility of $\{T_{\A}(t)\}_{t\geq 0}$ \textit{implies} the presence of
a spectral gap. It follows that under the conditions \eqref{Condition 1}-\eqref{Condition 2}-\eqref{Condition 3}) $\{T_{\A}(t)\}_{t\geq 0}$ has an asynchronous
exponential growth (see Theorem \ref{Thm:Asynch}). 

4. We show that once $\{T_{\A}(t)\}_{t\geq 0}$ has a spectral gap (i.e. once \eqref{Condition 4} is satisfied) the peripheral spectrum of $\mathcal{A}$
reduces to $s(\mathcal{A})$, i.e. 
\begin{equation*}
\sigma (\mathcal{A})\cap \{\lambda \in \mathbb{C}:\Re (\lambda )=s(\mathcal{A%
})\}=\{s(\mathcal{A})\},
\end{equation*}
and there exists a nonzero finite rank projection $P_{0}$ on $\mathcal{X}$
such that 
\begin{equation*}
\lim_{t\rightarrow \infty }\Vert e^{-s(\mathcal{A})t}T_{\A}(t)-e^{tD}P_{0}\Vert _{\L (\mathcal{X})}=0
\end{equation*}
where $D:=(s(\mathcal{A})-\mathcal{A})P_{0}$, (see Theorem \ref{Thm:No_Irr}). A priori, if $\{T_{\A}(t)\}_{t\geq 0}$ is \textit{not} irreducible then $P_{0}$ need not be
one-dimensional and the nilpotent operator $D$ need not be zero.

5. When $\{T_{\A}(t)\}_{t\geq 0}$ is \textit{not} irreducible but has a
spectral gap, it may happen that there exists a \textit{subspace} of $\mathcal{X}$ which is invariant under $\{T_{\A}(t)\}_{t\geq 0}$ and on which $\{T_{\A}(t)\}_{t\geq 0}$ exhibits an asynchronous exponential growth (see
Theorem \ref{Thm:Small_AEG}). 

The last two statements (Theorem \ref{Thm:No_Irr} and Theorem \ref{Thm:Small_AEG}) appear here for the first time. We deal also with the case
$$m=\infty$$
which has never been dealt with before. Its analysis is quite different from the
previous one:

6. The criterion of irreducibility is similar to the case $m<\infty$ (see
Theorem \ref{Thm:Irr_inf}).

7. As for $m<\infty $, we show that $\{T_{\A}(t)\}_{t\geq 0}$ has a spectral gap (i.e. $\omega_{\ess}(\A)<\omega_0(\A)$) if and only if
\begin{equation}\label{Condition 5}
s(\mathcal{B})<s(\mathcal{A})
\end{equation}
(see Theorem \ref{Thm:AEG_inf} and Remark \ref{Rem:Spec_gap-necessary}) where 
\begin{equation*}
\mathcal{B}=\mathcal{A}-B_{3}
\end{equation*}
and
\begin{equation*}
B_{3}
\begin{pmatrix}
u_{1} \\ 
u_{2}
\end{pmatrix}
=
\begin{pmatrix}
\int_{0}^{\infty }\beta (\cdot ,y)u_{1}(y)dy \\ 
0
\end{pmatrix}.
\end{equation*}
However, the condition \eqref{Condition 5} is much more delicate to check than in the finite case $(m<\infty)$.  Indeed, in the latter case, the fact that $s(\B)=-\infty$ (see Theorem \ref{Thm:SpectBound_inf}), implies $\omega_{\ess}(\{T_{\A}(t)\}_{t\geq 0})=-\infty$ and then \eqref{Condition 5} follows from an irreducibility argument (see Theorem \ref{Thm:SpectBound}) because the generator has a compact resolvent. The argument fails when $m=\infty$ and condition \eqref{Condition 5} need not hold in general even if $\{T_{\A}(t)\}_{t\geq 0}$ is irreducible. To check the condition \eqref{Condition 5} when $m=\infty$, a separate study of the spectral bounds $s(\B)$ and $s(\A)$ is necessary. Indeed: 

8. We show first that the \textit{real} spectrum of $\mathcal{B}$ is \textit{connected}
\begin{equation*}
\sigma (\mathcal{B})\cap \R =\left( -\infty, s(\mathcal{B})\right] 
\end{equation*}
and
\begin{equation*}
-\lim \sup_{x\rightarrow \infty }\mu (x)\leq s(\mathcal{B})\leq 0
\end{equation*}
(see Theorem \ref{Thm:Spec_Dtilde}). We can compute explicitly $s(\mathcal{B})$ if $c_{2}(\cdot)$ is a constant function and $\lim_{x\rightarrow \infty} \mu (x)$, $
\lim_{x\rightarrow \infty }c_{1}(x)$ exist (see Theorem \ref{Thm:SpecBound_Pol}).

9. We show that if 
\begin{equation*}
\int_{0}^{\infty }\beta (s,y)ds\geq \mu (y),\quad \forall y\geq 0
\end{equation*}
and 
\begin{equation*}
\liminf_{x\rightarrow \infty }\mu (x)>0,\qquad \liminf_{x\rightarrow \infty
}c_{2}(x)>0
\end{equation*}
then $s(\mathcal{A})\geq 0$ and $s(\mathcal{B})<0$ (see Theorem \ref{Thm:Conserv_Gap}). In particular $\{T_{\A}(t)\}_{t\geq 0}$ has a spectral gap, i.e. $\omega_{\ess}(\A)<\omega_0(\A)$.

10. We show also a ``converse" statement: if 
\begin{equation*}
\int_{0}^{\infty }\beta (s,y)ds\leq \mu (y),\quad \forall y\geq 0
\end{equation*}
and
\begin{equation*}
\lim_{x\rightarrow \infty }c_{2}(x)=0 \text{ or } \lim_{x\rightarrow \infty
}\mu (x)=0
\end{equation*}
then $s(\mathcal{B})=s(\mathcal{A})=0$ (see Theorem \ref{Thm:NoSpecGap}). In particular $\{T_{\A}(t)\}_{t\geq 0}$ has not a spectral gap, i.e. $\omega_{\ess}(\A)=\omega_0(\A)$.

11. Finally, we show that if $c_{1},c_{2}$ and $\mu$ are positive
constants and if $\beta _{1}(s):=\inf_{y\geq 0}\beta (s,y)$ is not trivial
then $s(\mathcal{A})>s(\mathcal{B})$, (see Theorem \ref{Thm:Partic_Case}); we can even
provide an explicit lower bound of the spectral gap $s(\mathcal{A})-s(\mathcal{B})$, (see Remark \ref{Rem:Bound_SpecGap}).

Some useful conjectures are also given, see Remark \ref{Rem:Lods}. The authors thank two anonymous referees for their useful remarks and suggestions.

\section{Models with bounded sizes}

\subsection{Framework and hypotheses}\label{Sec:Hyp}

In order to analyse the problem described by \eqref{Eq:Model}-\eqref{Eq:BoundCond}, we define the Banach space
$$\X=(L^1(0,m)\times L^1(0,m), \|.\|_\X)$$
endowed with the norm
$$\|(u_1,u_2)\|_{\X}=\|u_1\|_{L^1(0,m)}+\|u_2\|_{L^1(0,m)}.$$
We denote by $\X_+$ the nonnegative cone of $\X$ and we suppose in all this section the following hypotheses on the different parameters:
\begin{enumerate}
\item $\mu, c_1, c_2 \in L^\infty(0,m)$ and $\gamma_1, \gamma_2 \in W^{1,\infty}(0,m)$,
\item $\beta, \mu, c_1, c_2 \geq 0$ and there exists $\gamma_0>0$ such that for every $s\in[0,m]$, $\gamma_1(s)\geq \gamma_0, \gamma_2(s)\geq \gamma_0 $,
\item the operator
$$K:L^1(0,m)\ni u\mapsto \int_0^m \beta(\cdot,y)u(y)dy \in L^1(0,m)$$
is bounded.
\end{enumerate}
\begin{remark}\label{Rem:Bounded}
The integral operator $K$ is bounded if and only if
$$\sup_{y\in(0,m)} \int_0^m \beta(s,y)ds<\infty.$$
The same remark holds when $m=\infty$.
\end{remark}
Using \eqref{Eq:Model}, we define the operator
\begin{equation*}
\begin{array}{rcl}
\A\begin{pmatrix}
u_1\\
u_2
\end{pmatrix}&=&A\begin{pmatrix}
u_1 \\
u_2
\end{pmatrix}+B\begin{pmatrix}
u_1 \\
u_2
\end{pmatrix}\\
&=&\begin{pmatrix}
-(\gamma_1 u_1)' \\
-(\gamma_2 u_2)'
\end{pmatrix}+\begin{pmatrix}
-(\mu+c_1)u_1+c_2u_2+\int_0^m \beta(\cdot,y)u_1(y)dy) \\
-c_2 u_2+c_1 u_1
\end{pmatrix},
\end{array}
\end{equation*}
with domain
$$D(A)=\{(u_1,u_2)\in W^{1,1}(0,m)\times W^{1,1}(0,m): u_1(0)=0, u_2(0)=0\},$$
where $W^{1,1}(0,m)$ is the Sobolev space
$$W^{1,1}(0,m)=\{u\in L^1(0,m), u'\in L^1(0,m)\}.$$
We decompose $B$ into three bounded operators:
\begin{equation*}
\begin{array}{rcl}
B\begin{pmatrix}
u_ 1\\
u_2
\end{pmatrix}
&=&B_1\begin{pmatrix}
u_1 \\
u_2
\end{pmatrix}
+B_2\begin{pmatrix}
u_1 \\
u_2
\end{pmatrix}
+B_3\begin{pmatrix}
u_1 \\
u_2
\end{pmatrix}\\
&=&\begin{pmatrix}
-(\mu+c_1)u_1 \\
-c_2 u_2
\end{pmatrix} 
+\begin{pmatrix}
c_2 u_2 \\
c_1 u_1
\end{pmatrix}
+\begin{pmatrix}
\int_0^m \beta(\cdot,y)u_1(y)dy \\
0
\end{pmatrix}.
\end{array}
\end{equation*}
We are then concerned with the following Cauchy problem
\begin{equation*}
\left\{
\begin{array}{rcl}
U'(t)&=&\A U(t), \\
U(0)&=&(u^0_1,u^0_2)\in \X,
\end{array}
\right.
\end{equation*}
where
$$U(t)=(u_1(t),u_2(t))^T.$$

\subsection{Semigroup generation}

It is easy to prove:
\begin{lemma}\label{Lemma:Resolv}
Let $H=(h_1,h_2)\in \X$, $\lambda\in \R$ and $U=(\lambda I-A)^{-1}H:=(u_1,u_2)\in D(A)$. We have
\begin{equation}\label{Eq:Range}
\left\{
\begin{array}{rcl}
u_1(s)&=& \displaystyle \dfrac{1}{\gamma_1(s)} \int_0^s h_1(y)\exp\left(-\int_y^s \dfrac{\lambda}{\gamma_1(z)}dz\right)dy \vspace{0.1cm}, \\
u_2(s)&=& \displaystyle \dfrac{1}{\gamma_2(s)} \int_0^s h_2(y)\exp\left(-\int_y^s \dfrac{\lambda}{\gamma_2(z)}dz\right)dy,
\end{array}
\right.
\end{equation}
for every $s\in[0,m]$. In particular, $s(A)=-\infty$ and for every $(h_1,h_2)\in \X_+$,
$$\supp u_1=[\inf \supp h_1,m], \quad \supp u_2=[\inf \supp h_2,m],$$
where $\supp(f)$ refers to the support of a function $f$ and $\inf \supp(f)$ is its lower bound.
\end{lemma}
Note that if $h_i\geq 0$, then $u_i(x)>0$ if and only if $x>\inf \supp h_i$.
\begin{theorem}\label{Thm:Generation}
The operator $\A$ generates a $C_0$-semigroup $\{T_\A(t)\}_{t\geq 0}$ of bounded linear operators on $\X$.
\end{theorem}

\begin{proof}
Since $B$ is bounded, it suffices to prove that $A$ generates a contraction semigroup. We easily see that $D(A)$ is densely defined in $\X$. Moreover, for $\lambda\in \R$, the range condition 
$$(\lambda I-A)U=H,$$
with $U=(u_1,u_2)$ and $H=(h_1,h_2)\in \X$, is straightforward since $(u_1,u_2)$ is given by \eqref{Eq:Range}, so 
$$\|u_i\|_{L^1(0,m)}\leq \dfrac{m \|h_i\|_{L^1}}{\gamma_0}\exp\left(\dfrac{|\lambda|m}{\gamma_0}\right)<\infty$$
and 
$$\|u'_i\|_{L^1(0,m)}\leq \dfrac{(|\lambda|+\|\gamma'_i\|_{L^\infty})\|u_i\|_{L^1}+\|h_i\|_{L^1}}{\gamma_0}<\infty$$
for every $i\in\{1,2\}$, hence $U\in D(A)$.
It remains to prove that $A$ is a dissipative operator. Let $\lambda>0$, $U=(u_1,u_2)\in D(A)$, $H=(\lambda I-A)U$ and $H=(h_1,h_2)$. We prove that
$$\|H\|_\X\geq \lambda \|U\|_\X$$
i.e.
$$\|h_i\|_{L^1(0,m)}\geq \lambda \|u_i\|_{L^1(0,m)}, \ \forall i\in\{1,2\}.$$
By definition, we have $u_i(0)=0$ and
$$\lambda u_i(s)+(\gamma_i u_i)'(s)=h_i(s), \ \forall s\in (0,m].$$
We multiply the latter equation by $\sign(u_i(s))$ then integrate between $0$ and $m$. We get
$$\lambda \|u_i\|_{L^1(0,m)}+\int_0^m (\gamma_i u_i)'(s)\sign(u_i(s))ds=\int_0^m h_i(s)\sign(u_i(s))ds.$$
Any nonempty open set of the real line is a finite or countable union
of \textit{disjoints} open intervals (see \cite{Apostol74} Theorem
3.11, p. 51) so 
\begin{align*}
\{u_i>0\}& =\{s\in \left( 0,m\right) :u_i(s)>0\}=\underset{i\in \N }{\cup 
}(a_{i,1},a_{i,2}), \\
\{u_i<0\}& =\{s\in \left( 0,m\right) :u_i(s)<0\}=\underset{i\in \N }{\cup 
}(b_{i,1},b_{i,2}).
\end{align*}%
Since $u_i\in W^{1,1}(0,m)\hookrightarrow C([0,m])$ then $\forall i,j\in 
\N :u_i(a_{i,1})=0,$ $u_i(a_{i,2})=0,$ $u_i(b_{j,1})=0$ and $u_i(b_{j,2})=0$
(except possibly at $m$). Thus 
\begin{eqnarray*}  \label{Eq:Maj_Gamma}
&&\int_{0}^{m}(\gamma_i u_i)^{\prime }sign(u_i) =\int_{\left\{ u_i>0\right\}
}(\gamma_i u_i)^{\prime }-\int_{\left\{ u_i<0\right\} }(\gamma_i u_i)^{\prime } 
\notag \\
&=&\sum_{j\in 
\N  
} \left[\gamma_i (a_{j,2})u_i(a_{j,2})-\gamma_i (a_{j,1})u_i(a_{j,1})
\right]-\sum_{j\in 
\N 
} \left[\gamma_i (b_{j,2})u_i(b_{j,2})-\gamma_i (b_{j,1})u_i(b_{j,1})\right]
\notag \\
&=&\gamma_i (m)\left\vert u_i(m)\right\vert\geq 0.
\end{eqnarray*}
Hence
$$ \lambda \|u_i\|_{L^1}\leq \lambda\|u_i\|_{L^1}+\gamma_i(m)|u_i(m)|= \int_0^m h_i(s)\sign(u_i(s))ds\leq \|h_i\|_{L^1}$$ and we get the dissipativity of $A$. \\
Thus $A$ generates a contraction $C_0$-semigroup $\{T_\A(t)\}_{t\geq 0}$ by Lumer-Phillips Theorem (see \cite{Pazy83} Theorem 4.3, p. 14). Finally, as bounded perturbations of $A$, the operators $A+B_1$, $A+B_1+B_2$ and $\A$ generate quasi-contraction $C_0$-semigroups $\{T_{A+B_1}(t)\}_{t\geq 0}$, $\{T_{A+B_1+B_2}(t)\}_{t\geq 0}$ and $\{T_\A(t)\}_{t\geq 0}$ respectively (that is to say there exists $\omega_0 \geq 0$ such that $\Vert
\{T_{\mathcal{A}}(t)\}_{t\geq 0}\Vert_{\X} \leq e^{\omega_0 t}$, for every $t\geq 0$).
\end{proof}

\subsection{On positivity}

The time asymptotics of $\{T_\A(t)\}_{t\geq 0}$ is related to irreducibility arguments. We remind first some definitions and
results about positive and irreducible operators. We denote by $\langle
\cdot,\cdot\rangle $ the duality pairing between $\X$ and $\X^\prime.$

\begin{definition}\label{Def:Irr}
\mbox{}

\begin{enumerate}
\item For $f\in \X$, the notation $f>0$ means $f\in \X_+$
and $f\neq 0$.

\item An operator $O\in L(\X)$ is said to be positive if $Of\in \X_+$ for any $f\in \X_+$. We note this by $O\geq 0.$

\item A $C_{0}$-semigroup $\{T(t)\}_{t\geq 0}$ on $\X$ is said to be
positive if each operator $T(t)$ is positive.

\item A positive operator $O\in L(\X)$ is said to be positivity
improving if for every $f\in \X$, $f>0$ and every $f'\in \X'$, $f'>0$, we have $\langle
Of,f'\rangle >0$.

\item A positive operator $O\in L(\X)$ is said to be irreducible if
for every $f\in \X$, $f>0$ and every $f'\in \X'$, $f'>0$ there exists an integer $n$ such that $\langle O^{n}f,f'\rangle >0$.

\item A $C_{0}$-semigroup $\{T(t)\}_{t\geq 0}$ on $\X$ is said to be
irreducible if for every $f\in \X$, $f>0$ and every $f' \in \X'$, $f'>0$ there exists $t>0$ such
that $\langle T(t)f,f'\rangle >0$.

\item A subspace $\mathcal{Y}$ of $\X$ is said to be an ideal if $|f|\leq |g|$ and $g\in \mathcal{Y}$ imply $f\in \X$ where $| \cdot |$ denotes the \textit{absolute value}.
\end{enumerate}
\end{definition}

We recall that a $C_{0}$-semigroup $\{T(t)\}_{t\geq 0}$ on $\X$
with generator $\A$ is positive if and only if, for $\lambda $ large enough,
the resolvent operator $(\lambda I-\A)^{-1}$ is positive (see e.g. \cite{Clement87}, p. 165). We recall also
that a positive $C_{0}$-semigroup $\{T(t)\}_{t\geq 0}$ on $\X$ with generator $\A$ is irreducible if and only if, for $\lambda $ large enough, the resolvent operator $
(\lambda I-\A)^{-1}$ is positivity improving, if and only if, for $\lambda$ large enough, there is no closed ideal of $\X$ (except $\X$ and $\{0\}$) which is invariant under $(\lambda-\A)^{-1}$ (see \cite{Nagel86} C-III, Definition 3.1, p. 306).

\begin{definition}
For a closed operator $\A:D(\A)\subset \X\rightarrow \X$, we
denote by $\sigma (\A)$ its spectrum, $\rho(\A)$ its resolvent set and $s(\A)$ its spectral bound defined by

\begin{equation*}
s(\A):=\begin{cases}
\sup \left\{\Re(\lambda) ; \lambda \in \sigma(\A)\right\} & \mbox{if } \sigma(\A)\neq \emptyset, \\ 
-\infty & \mbox{if }\sigma(\A)=\emptyset.
\end{cases}
\end{equation*}
\end{definition}
We recall the following result which is a particular version of \cite{Voigt89}, Theorem 1.1.
\begin{lemma}\label{Lemma:Voigt}
Let $\A$ be a resolvent positive operator in $\X$ and $B\in \L(\X)$ a positive operator. We have
\begin{equation}\label{Eq:Voigt_Sum}
(\lambda-\A-B)^{-1}=(\lambda-\A)^{-1}\overset{\infty}{\underset{n=0}{\sum}}(B(\lambda-\A)^{-1})^n
\end{equation}
for every $\lambda>s(\A+B)$ and
\begin{equation}\label{Eq:Voigt}
s(\A+B)=\inf\{\lambda>s(\A): r_\sigma(B(\lambda-\A)^{-1})<1\}.
\end{equation}
\end{lemma}
Here $r_\sigma(\cdot)$ refers to the spectral radius. We introduce the following assumptions
\begin{equation}
\label{Eq:H1}
\forall \varepsilon\in(0,m), \quad \int_0^\varepsilon \int_{\varepsilon}^m \beta(s,y)dyds>0,
\end{equation}
\begin{equation}
\label{Eq:H2}
\inf \text{supp } c_1=0,
\end{equation}
\begin{equation}
\label{Eq:H3}
\sup \text{supp } c_2=m.
\end{equation}

\begin{theorem}\label{Thm:Irr}
The $C_0$-semigroup $\{T_\A(t)\}_{t\geq 0}$ is  irreducible if and only if the assumptions \eqref{Eq:H1}-\eqref{Eq:H2}-\eqref{Eq:H3} are satisfied.
\end{theorem}

\begin{proof}
\begin{enumerate}
\item Note first that the semigroup $\{T_\A(t)\}_{t\geq 0}$ is positive. Indeed, using Lemma \ref{Lemma:Resolv}, we readily see that the semigroup $\{T_{\A}(t)\}_{t\geq 0}$ is positive since $(\lambda I-A)^{-1}$ is positive for every $\lambda>-\infty$. Since $B_1$ is a bounded operator and
$$B_1+\|B_1\|I\geq 0,$$
then it follows (see e.g. \cite{Nagel86} Theorem 1.11, C-II, p. 255) that $\{T_{A+B_1}(t)\}_{t\geq 0}$ is positive. Finally, since $B_2$ and $B_3$ are positive operators, then the $C_0$-semigroups $\{T_{A+B_1+B_2}(t)\}_{t\geq 0}$ and $\{T_\A(t)\}_{t\geq 0}$ are also positive.

\item Now we suppose that the assumptions \eqref{Eq:H1}-\eqref{Eq:H2}-\eqref{Eq:H3} are satisfied and we prove that $(\lambda I-\A)^{-1}$ is positivity improving for $\lambda$ large enough. Actually, since $B_1+\|B_1\| I\geq 0$, we have
\begin{eqnarray*}
(\lambda I-\A)^{-1}&=&((\lambda+\|B_1\|)I-A-(B_1+\|B_1\|I)-B_2-B_3)^{-1} \\
&\geq& ((\lambda+\|B_1\|)I-A-B_2-B_3)^{-1}
\end{eqnarray*}
so it suffices to show that $(\lambda I-A-B_2-B_3)^{-1}$ is positivity improving for $\lambda$ large enough. \\
Using \eqref{Eq:Voigt_Sum}, we first see that
\begin{eqnarray}
&&\left(\lambda I-A-B_2-B_3\right)^{-1}=\left(\lambda I-A-B_2\right)^{-1}\sum_{n=0}^{\infty }\left(B_3\left(\lambda I-A-B_2\right)^{-1}\right)^n \nonumber \\
&& \qquad =(\lambda I-A)^{-1}\sum_{l=0}^{\infty}\left(B_2(\lambda I-A)^{-1}\right)^l\sum_{n=0}^{\infty }\left(B_3(\lambda I-A-B_2)^{-1}\right)^n. \label{Eq:Resolv}
\end{eqnarray}
Since we have
$$\sum_{l=0}^\infty \left(B_2(\lambda I-A)^{-1}\right)^l\geq I+B_2(\lambda I-A)^{-1}$$
then we get
\begin{eqnarray*}
&&\sum_{n=0}^{\infty }\left(B_3(\lambda I-A-B_2)^{-1}\right)^n \\
&\geq& \sum_{n=1}^{\infty }\left(B_3(\lambda I-A-B_2)^{-1}\right)^{n-1}B_3(\lambda I-A-B_2)^{-1} \\
&\geq& \sum_{n=1}^{\infty }\left(B_3(\lambda I-A)^{-1}\right)^{n-1}B_3(\lambda I-A)^{-1}\sum_{l=0}^\infty \left(B_2(\lambda I-A)^{-1}\right)^l \\
&\geq& \sum_{n=1}^{\infty }\left(B_3(\lambda I-A)^{-1}\right)^{n}(I+B_2(\lambda I-A)^{-1}).
\end{eqnarray*}
Consequently we have
\begin{eqnarray*}
& &(\lambda I-A-B_2-B_3)^{-1}\\
\geq & & (\lambda I-A)^{-1}(I+B_2(\lambda I-A)^{-1})\sum_{n=1}^\infty(B_3(\lambda I-A)^{-1})^n (I+B_2(\lambda I-A)^{-1}).
\end{eqnarray*}
Let $U:=(u_1,u_2)=(\lambda I-A-B_2-B_3)^{-1}H$ with $H\in \X_+$.
Let us show that 
\begin{equation*}
u_1(s)>0, \quad u_2(s)>0 \quad \text{ a.e.}
\end{equation*}
once 
\begin{equation*}
H=(h_1, h_2)\in \X_{+}-\left\{ 0\right\}.
\end{equation*}
Step 1: we start by proving that 
\begin{equation}\label{Eq:Irr_Step1}
\forall H\in \X_+-\left\{0\right\}, \exists \ h\in L^1_+(0,m)-\left\{0\right\}:(I+B_2(\lambda I-\A)^{-1})H\geq (h,0).
\end{equation}
If $H:=(h_1,0)$, then it is clear that \eqref{Eq:Irr_Step1} is satisfied, by taking $h=h_1$.
If $H:=(0,h_1)$, then, using Lemma \ref{Lemma:Resolv}, we get
$$(\lambda I-A)^{-1}H=:(0,h_2)\in D(A)$$
where
$$\supp h_2=[\inf \supp h_1,m].$$
By assumption \eqref{Eq:H3}, we have
$$ |\supp c_2 \cap \supp h_2| \neq 0$$
where $|I|$ denotes the Lebesgue measure of an interval $I$. Thus
$$B_2(\lambda I-A)^{-1}H=(c_2 h_2,0)$$
and \eqref{Eq:Irr_Step1} is satisfied with $h=c_2h_2$. In any case it suffices to show that
$$(\lambda I-A)^{-1}(I+B_2(\lambda I-A)^{-1})\sum_{n=1}^\infty(B_3(\lambda I-A)^{-1})^n H>0 \quad \text{a.e.}$$
for every $H=(h,0)\in \X_+ -\left\{0\right\}$. We have $(\lambda I-A)^{-1}H=(h_1,0)\in D(A)$,
with
$$ \supp h_1=[\inf \supp h,m].$$

Step 2: now we prove that for every $H:=(h,0)\in \X_+-\left\{0\right\}$, then
\begin{equation}\label{Eq:Irr_Step2}\left( \sum_{n=1}^\infty(B_3(\lambda I-A)^{-1})^n \right)H=:(\tilde{h},0) \quad \text{where} \quad \inf \supp(\tilde{h})=0.
\end{equation}
Let $H:=(h_1,0)\in \X_+-\left\{0\right\}$, then
$$\sum_{n=1}^\infty (B_3(\lambda I-A)^{-1})^n H=:(h_2,0).$$
Suppose by contradiction that
$$k:=\inf \supp h_2>0.$$
Using Lemma \ref{Lemma:Resolv}, we get
$$(\lambda I-A)^{-1}(h_2,0)=:(h_3,0),$$
with $\supp h_3=[k,m]$ and we have
$$B_3(\lambda I-A)^{-1}(h_2,0)=:(h_4,0).$$
If 
\begin{equation}\label{Eq:Contrad}
\tilde{k}:=\inf \supp h_4<k
\end{equation}
holds, then we get a contradiction by definition of $k$ and \eqref{Eq:Irr_Step2} is satisfied. So it remains to prove \eqref{Eq:Contrad}. Suppose by contradiction that $\tilde{k}\geq k$, then we get $h_4\equiv 0$ on $[0,k]$ and
$$\int_{k}^m \beta(s,y)h_3(y)dy\leq \int_0^m \beta(s,y)h_3(y)dy=h_4(s)=0 \ \text{a.e. } s\in[0,k].$$
Moreover, since $h_3(y)>0$ a.e. $y\in(k,m]$, we would get
$$\int_{k}^m \beta(s,y)dy=0, \ \text{a.e. }s\in [0,k]$$
which contradicts Assumption \eqref{Eq:H1}. 

Step 3: we finally prove that
\begin{equation}\label{Eq:Irr_Step3}
(\lambda I-A)^{-1}\left(I+B_2(\lambda I-A)^{-1}\right)H>0 \quad \text{ a.e}
\end{equation}
for every $H=(h,0)\in \X-\left\{0\right\}$ such that $\inf \supp h=0.$

Using Lemma \ref{Lemma:Resolv} we have
$$(\lambda I-A)^{-1}H=(h_1,0),$$
where $h_1(s)>0$ for every $s\in (0,m]$.
Using Assumption \eqref{Eq:H1} we get
$$B_2(\lambda I-A)^{-1}H=B_2(h_1,0)=:(0,h_2),$$
where $h_2:=c_1h_1$ satisfies
$$\inf \supp h_2=0.$$
Once again with Lemma \ref{Lemma:Resolv}, we get
$$(\lambda I-A)^{-1}(0,h_2)=:(0,h_3),$$
where $h_3(s)>0$ for every $s\in (0,m]$. Finally
$$(u_1,u_2):=U=(\lambda I-A)^{-1}(I+B_2(\lambda I-A)^{-1})H\geq (h_1,h_3)$$
so
\begin{equation*}
u_1(s)>0, \quad u_2(s)>0 \quad \text{ a.e.}
\end{equation*}
and $\{T_\A(t)\}_{t\geq 0}$ is irreducible.

\item Now, to prove the converse, we use the contraposition. We suppose that either \eqref{Eq:H1}, \eqref{Eq:H2} or \eqref{Eq:H3} is not satisfied. In each case, we exhibit a nontrivial closed ideal of $\X$ that is invariant under $(\lambda I-\A)^{-1}$, which implies that the $C_0$-semigroup $\{T_{\A}(t)\}_{t\geq 0}$ is not irreducible.
\begin{enumerate}
\item Suppose that \eqref{Eq:H1} does not hold, then 
\begin{equation}\label{Eq:H1-notverif}
\exists \ \varepsilon \in (0,m): \int_0^\varepsilon \int_\varepsilon^m \beta(s,y)dyds=0
\end{equation}
i.e.
$$\beta(s,y)=0 \quad \text{a.e. } \ s<\varepsilon<y.$$
We identify $L^1(\varepsilon,m)$ to the closed subspace of $L^1(0,m)$ of functions vanishing a.e. on $(0,\varepsilon)$. Let $\lambda>s(\A)$, we want to prove that
$$\Y:=L^1(\varepsilon,m)\times L^1(\varepsilon,m)$$
is a closed ideal of $\X$ that is invariant under $(\lambda I-\A)^{-1}$. Since $B_1\leq 0$, we have
\begin{equation}\label{Eq:Resolv_maj}
(\lambda I -\A)^{-1}\leq (\lambda I-(A+B_2+B_3))^{-1}
\end{equation}
where the latter resolvent is given by \eqref{Eq:Resolv}. Using Lemma \ref{Lemma:Resolv} we see that $\Y$ is invariant under $(\lambda I-A)^{-1}$. It is also clear that $\Y$ is invariant under $B_2$ and consequently also under $(\lambda I-(A+B_2))^{-1}$ by using \eqref{Eq:Voigt_Sum}. It remains to prove that $\Y$ is invariant under $B_3$. Let
$$H:=(h_1,h_2)\in \Y, \qquad B_3H=:(u,0),$$
where
$$u(s)=\int_0^m \beta(s,y)h_1(y)dy=\int_\varepsilon^m \beta(s,y)h_1(y)dy=0 \qquad \text{a.e.} \quad s\in [0,\varepsilon]$$
by Assumption \eqref{Eq:H1-notverif}. Thus $\Y$ is invariant under $B_3$ and consequently under $(\lambda I-(A+B_2+B_3))^{-1}$ by using \eqref{Eq:Voigt_Sum}. Finally, $\Y$ is invariant under $(\lambda I-\A)^{-1}$ by using \eqref{Eq:Resolv_maj}.

\item Suppose that \eqref{Eq:H2} does not hold. Let $\lambda>s(\A)$ and
$$k:=\inf \supp c_1>0.$$
We want to prove that
$$\Y:=L^1(0,m)\times L^1(k,m)$$
is a closed ideal of $\X$ that is invariant under $(\lambda I-\A)^{-1}$. Let $H:=(h_1,h_2)\in \Y$.
Using \eqref{Eq:Resolv_maj}, we have
$$(\lambda I-\A)^{-1}H\leq (\lambda I-(A+B_2+B_3))^{-1}H=:(u_1,u_2)$$
where $(u_1, u_2)\in D(A)$ satisfy 
\begin{equation*}
\left\{
\begin{array}{lll}
\lambda u_2(s)+(\gamma_2 u_2)'(s)-c_1(s)u_1(s)=h_2(s) \ \text{a.e. } s\in[0,m], \\
u_2(0)=0.
\end{array}
\right.
\end{equation*}
We then get
$$\lambda u_2(s)+(\gamma_2 u_2)'(s)=0 \text{ a.e. } s\in [0,k]$$
which lead to
$$u_2\equiv 0 \text{ on } [0,k].$$
Consequently $\Y$ is invariant under $(\lambda I-(A+B_2+B_3))^{-1}$ and under $(\lambda I-\A)^{-1}$ using \eqref{Eq:Resolv_maj}.

\item Suppose that \eqref{Eq:H3} does not hold. Let $\lambda>s(\A)$ and
$$k:=\sup \supp c_2<m.$$
We want to prove that
$$\Y:=\{0\}\times L^1(k,m)$$
is a closed ideal of $\X$ that is invariant under $(\lambda I-\A)^{-1}$. Using Lemma \ref{Lemma:Resolv}, we see that $\Y$ is invariant under $(\lambda I-A)^{-1}$. Moreover, let $H:=(0,h_1)\in \Y$, then we have
$$B_2 H=(c_2h_1,0)=(0,0)$$
since 
$$\supp (c_2) \cap \supp (h_1) = \emptyset.$$
Consequently, $\Y$ is invariant under $B_2$. It remains to prove that it is also invariant under $B_3$. But this is obvious since
$$B_3 H=(0,0).$$
Consequently, $\Y$ is invariant under $(\lambda I-(A+B_2+B_3))^{-1}$ and $(\lambda I-\A)^{-1}$ by using \eqref{Eq:Voigt_Sum}.
\end{enumerate}
\end{enumerate}
\end{proof}

We note that in \cite{Farkas2010}, the irreducibility is obtained under the assumptions \eqref{Eq:H2}-\eqref{Eq:H3} and the following one:
$$\exists \ \varepsilon_0>0 : \forall \varepsilon\in(0,\varepsilon_0], \qquad \int_0^\varepsilon \int_{m-\varepsilon}^m \beta(s,y)dyds>0.$$
In the continuous case, this latter assumption implies $\beta(0,m)>0$, so active cells of maximal size can produce offspring of minimal size. This is not necessary in our statement. The biological meaning of \eqref{Eq:H2}-\eqref{Eq:H3} is the following: active cells of minimal size can become quiescent, and quiescent cells of maximal size can become active.

\subsection{On the spectral bound}

We start with a useful
\begin{lemma}\label{Lemma:Volterra}
Let $k>0$ a positive constant and define the so-called Volterra operator $V:L^1(0,m)\to L^1(0,m)$ by
$$Vh(s)=k\int_0^s h(y)dy.$$
Then $r_\sigma(V)=0$ and $\sigma(V)=\{0\}$.
\end{lemma}

\begin{proof}
By induction, we can show that 
$$V^n h(s)=k^n \int_0^s h(y) \dfrac{(s-y)^{n-1}}{(n-1)!}dy,$$
for every $s\in[0,m]$, $n\geq 0$ and $h\in L^1(0,m)$. We then get
$$\|V^n\|\leq  \dfrac{k^n m^{n}}{n!}.$$
Consequently,
$$r_\sigma(V):=\lim_{n\to \infty}\|V^n\|^{1/n}\leq \lim_{n\to \infty}\dfrac{k m}{(n!)^{1/n}} = 0,$$
since
$$(n!)^{1/n}\approx \dfrac{n}{e}(\sqrt{2\pi n})^{1/n}$$
by Sterling's formula.
\end{proof}
We need also
\begin{lemma}\label{Lemma:Commute}
Let $V_1, V_2: L^1(0,m)\to L^1(0,m)$ two bounded operators. If $V_1V_2=V_2V_1$, then
$$r_\sigma(V_1V_2)\leq r_\sigma(V_1)r_\sigma(V_2).$$	
\end{lemma}

\begin{proof}
It is clear that 
\begin{eqnarray*}
r_\sigma(V_1 V_2)&=&\underset{n\to \infty}{\lim}\|(V_1 V_2)^n\|^{1/n}=\underset{n\to \infty}{\lim}\|V_1^n V_2^n\|^{1/n} \\
&\leq& \underset{n\to \infty}{\lim}\|V_1^n\|^{1/n} \|V_2^n\|^{1/n}=r_\sigma(V_1) r_\sigma(V_2),
\end{eqnarray*}
by using Gelfand's formula.
\end{proof}

Note that $\A$ has a compact resolvent (and consequently the
spectrum of $\A$ is composed (at most) of isolated eigenvalues with
finite algebraic multiplicity). This follows from the fact that the
canonical injection $i:(D(A),\Vert .\Vert _{D(A)})\rightarrow (\X,\Vert .\Vert _\X)$ is compact (\cite{Brezis1999} Theorem VIII.7, p. 129), and $D(A)=D(\A)$ since $B\in \L (\X)$ (see e.g. \cite{EngelNagel2000} Proposition II.4.25, p. 117). We are ready to show

\begin{theorem}\label{Thm:SpectBound_inf}
The spectrum of $A+B_1+B_2$ is empty and consequently $s(A+B_1+B_2)=-\infty$.
\end{theorem}

\begin{proof}
Let $\lambda>-\infty$ and define the operators
\begin{equation}\label{Eq:Op_Ai}
A^i_0u=-(\gamma_i u)', \quad \forall i\in\{1,2\}
\end{equation}
for every $u\in D(A^1_0)=D(A^2_0):=\{u\in W^{1,1}(0,m):u(0)=0\}$. Thus, using Lemma \ref{Lemma:Resolv}, we get
\begin{equation}\label{Eq:Maj_Resolv}
(\lambda I-A^i_0)^{-1}h(s)\leq k_i \int_0^s h(y)dy=:V_ih(s), \ \forall s\in[0,m], \ \forall i\in \{1,2\}, \ \forall h\in L^1_+(0,m)
\end{equation}
where $k_1$ and $k_2$ are some positive constants and $V_1$, $V_2$ are Volterra operators. We see that
$$B_2(\lambda I-A)^{-1}h\leq \tilde{B_2}\begin{pmatrix}
V_1 & 0\\
0 & V_2
\end{pmatrix}h, \quad \forall h\in \X_+,$$
since $A$ is resolvent positive, where 
\begin{equation}\label{Eq:Op_B2tilde} \tilde{B_2}\begin{pmatrix}
h_1 \\
h_2
\end{pmatrix}=\begin{pmatrix}
\|c_2\|_{L^\infty}h_2 \\
\|c_1\|_{L^\infty}h_1
\end{pmatrix}, \quad \forall (h_1,h_2)^T \in \X_+
\end{equation}
is a positive operator. The fact that $\tilde{B_2}$ and $(V_1,V_2)^T$ commute implies that
$$r_\sigma(\tilde{B_2}(\lambda I-A)^{-1})\leq r_\sigma(\tilde{B_2}) r_\sigma\begin{pmatrix}
V_1 & 0 \\
0 & V_2
\end{pmatrix}$$
using Lemma \ref{Lemma:Commute}. Since $V_1$ and $V_2$ are Volterra operators, then
$$r_\sigma\begin{pmatrix}
V_1 & 0 \\
0 & V_2
\end{pmatrix}=\max\{r_\sigma(V_1),r_\sigma(V_2)\}=0.$$
Consequently, we have
$$r_\sigma(B_2(\lambda I-A)^{-1})\leq r_\sigma(\tilde{B_2}(\lambda I-A)^{-1})=0$$
for every $\lambda>-\infty$ and 
$$s(A+B_2)=s(A)=-\infty$$ 
by using \eqref{Eq:Voigt} and Lemma \ref{Lemma:Resolv}. Finally, since $B_1\leq 0$, then we get
$$s(A+B_1+B_2)\leq s(A+B_2)=-\infty,$$
which ends the proof.
\end{proof}

On the other hand, $\sigma(\A)$ need not be empty. Indeed:

\begin{theorem}\label{Thm:SpectBound}
The spectrum of $\A$ is not empty, or equivalently, $s(\A)>-\infty$ if and only if
\begin{equation}\label{Eq:H1_bis}
\exists \ \delta \in(0,m) : \int_0^{\delta}\int_{\delta}^m \beta(s,y)dyds>0.
\end{equation}
\end{theorem}

\begin{proof}
\begin{enumerate}
\item Suppose that \eqref{Eq:H1_bis} is satisfied. By continuity argument, we can find $\delta_2\in(\delta,m)$ such that
\begin{equation}\label{Eq:H1_int}
\int_0^{\delta}\int_{\delta_2}^m \beta(s,y)dyds>0.
\end{equation}
Let $\lambda>s(\A)$ then
$$(\lambda-\A)^{-1}\geq(\lambda-(A+B_1+B_3))^{-1}=\begin{pmatrix}
(\lambda-(A^1_{\mu+c_1}+K))^{-1}\\
(\lambda-A^2_{c_2})^{-1}
\end{pmatrix}$$
since $B_2\geq 0$, where $A^1_{\mu+c_1}$ and $A^2_{c_2}$ are defined by
\begin{equation}\label{Eq:Op_AiTilde}
A^1_{\mu+c_1}u=-(\gamma_1 u)'-(\mu+c_1)u, \qquad A^2_{c_2}u=-(\gamma_2 u)'-c_2 u,
\end{equation}
and $D(A^1_{\mu+c_1})=D(A^2_{c_2})=D(A^1_0)$. Thus, we have
$$r_\sigma \left((\lambda-\A)^{-1}\right)\geq\max\left\{r_\sigma \left((\lambda-(A^1_{\mu+c_1}+K))^{-1}\right), r_\sigma\left((\lambda-A^2_{c_2})^{-1}\right)
\right\}.$$
It then suffices to show that
$$r_\sigma\left((\lambda-(A^1_{\mu+c_1}+K))^{-1}\right)>0.$$
First, we see that
$$(\lambda-(A^1_{\mu+c_1}+K))^{-1}\geq \left((\lambda+\|\mu\|_{L^\infty}+\|c_1\|_{L^\infty})I-(A^1_0+K)\right)^{-1},$$
so we just need to prove that for $\lambda$ large enough we have
$$r_\sigma\left((\lambda-(A^1_0+K))^{-1}\right)>0.$$
By \eqref{Eq:Voigt_Sum}, we know that
$$(\lambda-(A^1_0+K))^{-1}\geq (\lambda-A^1_0)^{-1}K(\lambda-A^1_0)^{-1}.$$
Let $v\in L^1(\delta, \delta_2)$, then using Lemma \ref{Lemma:Resolv}, we get
$$(\lambda-A^1_0)^{-1}v=:v_1,$$
where $v_1(s)>0$ for every $s\in(\inf \supp (v),m]$. In particular, we have
$$v_1(s)>0, \quad\forall s\in[\delta_2,m]$$
since $\inf \supp (v) \leq \delta_2$. Therefore we have
$$K(\lambda-A^1_0)^{-1}v=Kv_1=:v_2,$$
where $\inf \supp (v_2)\leq \delta$. Indeed, suppose by contradiction that
$$\inf \supp (v_2)>\delta,$$
then $v_2\equiv 0$ on $[0,\delta]$. We would have
$$\int_{\delta_2}^m \beta(s,y)v_1(y)dy\leq \int_0^m \beta(s,y)v_1(y)dy=v_2(s)=0, \quad \text{a.e.} \quad s\in [0,\delta],$$
and
$$\beta(s,y)=0, \qquad \text{a.e.} \quad s\in[\delta_2,m], \ y\geq \delta_2$$
since $v_1(s)>0$ for every $s\in[\delta_2,m]$, which contradicts \eqref{Eq:H1_int}. Define the function 
$$v_3:=(\lambda-A^1_0)^{-1}K(\lambda-A^1_0)^{-1}v=(\lambda-A^1_0)^{-1}v_2,$$
that satisfies
$$v_3(s)>0, \quad \forall s\in[\inf \supp (v_2),m]$$
by Lemma \ref{Lemma:Resolv}. In particular we have $v_3(s)>0$ for every $s\in[\delta, \delta_2]$. It implies that
\begin{equation}\label{Eq:Irr}
(\lambda-(A^1_0+K))^{-1}v(s)>0, \quad \forall s\in[\delta, \delta_2], \quad \forall v\in L^1(\delta,\delta_2),
\end{equation}
for $\lambda$ large enough. We also know that
$$(\lambda-(A^1_0+K))^{-1}\geq (\lambda-(A^1_0+K))^{-1}_{\mid L^1(\delta,\delta_2)}\geq \chi_{[\delta,\delta_2]}(\lambda-(A^1_0+K))^{-1}_{\mid L^1(\delta,\delta_2)},$$
where $\chi_{[\delta,\delta_2]}$ is the indicator function of $[\delta,\delta_2]$, so
$$r_\sigma \left((\lambda-(A^1_0+K))^{-1}\right)\geq r_\sigma\left(\chi_{[\delta,\delta_2]}(\lambda-(A^1_0+K))^{-1}_{\mid L^1(\delta,\delta_2)}\right).$$
Using \eqref{Eq:Irr} and the fact that $\A$ is resolvent compact, then the operator
$$\chi_{[\delta,\delta_2]}(\lambda-(A^1_0+K))^{-1}_{\mid L^1(\delta,\delta_2)}:L^1(\delta,\delta_2)\to L^1(\delta,\delta_2)$$
is compact and positivity improving. Consequently
$$r_\sigma\left(\chi_{[\delta,\delta_2]}(\lambda-(A^1_0+K))^{-1}_{\mid L^1(\delta,\delta_2)}\right)>0$$
(see \cite{Pagter86} Theorem 3) and
$$r_\sigma\left((\lambda-\A)^{-1}\right)>0.$$
Moreover, we know that
$$r_\sigma\left((\lambda-\A)^{-1}\right)=\dfrac{1}{\lambda-s\left(\A\right)}$$
(see \cite{Nagel86} Proposition 2.5, p. 67), so we get $s(\A)>-\infty$.

\item Now to prove the converse, we use the contraposition. Suppose that the assumption \eqref{Eq:H1_bis} is not satisfied, that is
\begin{equation}\label{Eq:Assump-Contrap}
\forall \ \delta \in(0,m): \int_0^\delta \int_\delta^m \beta(s,y)dyds=0
\end{equation}
i.e. 
$$\beta(s,y)=0, \qquad \text{a.e.} \quad s<y.$$
Suppose \textit{momentarily} that there exists a Volterra operator $V$ in $L^1(0,m)$ such that
\begin{equation}\label{Eq:Bounded_Volt}
(\lambda I-(A^1_0+K))^{-1}h(s)\leq Vh(s), \quad \forall s\in[0,m], \quad \forall h\in L^1_+(0,m),
\end{equation}
for every $\lambda>-\infty$, where $A^1_0$ is given by \eqref{Eq:Op_Ai}. We would have
$$r_\sigma\left((\lambda-(A^1_0+K))^{-1}\right)\leq r_\sigma(V)=0$$
and then
$$r_\sigma\left((\lambda I-(A+B_3))^{-1}\right)=r_\sigma\begin{pmatrix}
(\lambda I-(A^1_0+K))^{-1} \\
(\lambda I-A^2_0)^{-1}
\end{pmatrix}=0$$
since
$$r_\sigma\left((\lambda I-A^2_0)^{-1}\right)\leq r_\sigma\left((\lambda I-A)^{-1}\right)=0.$$
Consequently we have
$$s(A+B_3)=-\infty.$$
By assumption, we know that
$$B_2(\lambda I-(A+B_3))^{-1}\leq \tilde{B_2}\begin{pmatrix}
V & 0 \\
0 & V_2
\end{pmatrix},$$
where $V_2$ and $\tilde{B_2}$ are respectively defined by  \eqref{Eq:Maj_Resolv} and \eqref{Eq:Op_B2tilde}. The fact that $\tilde{B_2}$ and $(V,V_2)^T$ commute implies that
$$r_\sigma\left(\tilde{B_2}(\lambda I-(A+B_2))^{-1}\right)\leq r_\sigma(\tilde{B_2})r_\sigma\begin{pmatrix}
V & 0 \\
0 & V_2
\end{pmatrix}=0$$
using Lemma \ref{Lemma:Commute} and since $V$ and $V_2$ are Volterra operators.
Consequently, we have
$$r_\sigma(B_2(\lambda I-(A+B_2))^{-1})\leq r_\sigma(\tilde{B_2}(\lambda I-(A+B_2))^{-1})=0$$
for every $\lambda>-\infty$ and
$$s(A+B_2+B_3)=-\infty$$
by using \eqref{Eq:Voigt}. Finally we have
$$s(\A)\leq s(A+B_2+B_3)=-\infty$$
since $B_1\leq 0$.

Consequently it remains to prove \eqref{Eq:Bounded_Volt}. First, we know that
$$(\lambda-(A^1_0+K))^{-1}=(\lambda-A^1_0)^{-1}\sum_{n=0}^\infty (K(\lambda-A^1_0)^{-1})^n,$$
using \eqref{Eq:Voigt_Sum} for $\lambda$ large enough. Let $v\in L^1_+(0,m)$, then we have
\begin{equation*}
\begin{array}{rcl}
K(\lambda-A^1_0)^{-1}v(s)&\leq& k_1 \displaystyle \int_0^m \beta(s,y)\int_0^y v(z)dzdy \vspace{0.1cm}\\
&=& k_1 \displaystyle \int_0^s v(z)\int_z^s \beta(s,y)dydz, \quad \forall s\in[0,m],
\end{array}
\end{equation*}
using \eqref{Eq:Assump-Contrap}, where $k_1$ is defined in \eqref{Eq:Maj_Resolv}. We then get
\begin{equation*}
\begin{array}{rcl}
(\lambda-A^1_0)^{-1}K(\lambda-A^1_0)^{-1}v(s)&\leq& k_1^2 \displaystyle \int_0^s \int_0^y v(z) \int_z^y \beta(y,\xi)d\xi dz dy \vspace{0.1cm}\\
&\leq& k_1^2 k_\beta \displaystyle \int_0^s v(z) (s-z)dz, \quad \forall s\in[0,m],
\end{array}
\end{equation*}
where
\begin{equation}
\label{Eq:k_beta}
k_\beta=\sup_{y\in[0,m]}\int_0^m \beta(z,y)dz
\end{equation}
and
\begin{equation*}
\begin{array}{rcl}
(K(\lambda-A^1_0)^{-1})^2 v(s) \leq k_1^2 k_\beta \displaystyle \int_0^s \beta(s,y) \int_0^y v(z) (y-z)dzdy.
\end{array}
\end{equation*}
We then show by induction that
\begin{equation*}
\begin{array}{rcl}
(K(\lambda-A^1_0)^{-1})^n v(s) &\leq& k_1^n k_\beta^{n-1} \displaystyle \int_0^s \beta(s,y) \int_0^y v(z) \dfrac{(y-z)^{n-1}}{(n-1)!}dzdy, \\
&\leq& k_1\dfrac{(k_1 k_\beta m)^{n-1}}{(n-1)!}\displaystyle \int_0^s \beta(s,y) \int_0^y v(z) dzdy
\end{array}
\end{equation*}
for every $s\in[0,m]$ and every $n\geq 0$. Consequently, we get
\begin{equation*}
\begin{array}{rcl}
\underset{n\geq 1}{\sum}(K(\lambda-A^1_0)^{-1})^n v(s) \leq k_1 e^{k_1 k_\beta m} \displaystyle \int_0^s \beta(s,y) \int_0^y v(z) dzdy,
\end{array}
\end{equation*}
and then
\begin{equation*}
\begin{array}{rcl}
(\lambda-(A^1_0+K))^{-1}v(s) &\leq& k_1(1+m k_1 k_\beta e^{k_1 k_\beta m}) \displaystyle \int_0^s v(y)dy \\
&\leq& C \displaystyle \int_0^s v(y)dy=:Vv(s), \quad \forall s\in[0,m],
\end{array}
\end{equation*}
where $C>0$, for every $v\in L^1_+(0,m)$, which proves \eqref{Eq:Bounded_Volt}.
\end{enumerate}
\end{proof}

Note that Assumption \eqref{Eq:H1_bis} which characterizes that $s(A)>-\infty$ is much weaker than the assumptions in Theorem \ref{Thm:Irr} which characterize the irreducibility of the semigroup. Moreover, Theorem \ref{Thm:SpectBound} provides us with the existence of a real leading eigenvalue since $s(\A)\in \sigma(\A)$ (see e.g. \cite{Clement87} Theorem 8.7, p. 202). In \cite{Farkas2010}, the spectral gap is obtained under the assumption
\begin{equation}\label{Hyp:Gap_Farkas}
\beta \in \mathcal{C}([0,m]^2), \qquad \exists \ 0\leq s^*<y^*\leq m: \beta(s^*,y^*)>0.
\end{equation}
It is clear that \eqref{Hyp:Gap_Farkas} implies that \eqref{Eq:H1_bis} is satisfied.

\subsection{On asynchronous exponential growth}

Let us remind some definitions and results about \textit{asynchronous
exponential growth} (see \cite{EngelNagel2000}, \cite{Nagel86} and \cite{Webb87} for the details).

\begin{definition}\label{Def:Ess}
Let $\L (\X)$ be the space of bounded linear operators on $\X$ and let $\mathcal{K}(\X)$ be the subspace of compact
operators on $\mathcal{\X}$. The essential norm $\|L\|_{\ess}$ of $L\in \L (\X)$ is given by 
\begin{equation*}
\|L\|_{\ess}=\underset{K\in \mathcal{K}(\X)}{%
\inf }\|L-K\|_\X
\end{equation*}
(see e.g. \cite{EngelNagel2000} p. 249). Let $\{T(t)\}_{t\geq 0} $ be a $C_{0}$-semigroup on $\X$
with generator $\A:D(\A)\subset \X\rightarrow \X$. The
growth bound (or type) of $\{T(t)\}_{t\geq 0} $ is given by 
\begin{equation*}
\omega _{0}(\A)=\underset{t\rightarrow \infty }{\lim }\dfrac{\ln (\Vert
T(t)\Vert _{\X})}{t},
\end{equation*}
and the essential growth bound (or essential type) of $\{T(t)\}_{t\geq 0}$ is given by 
\begin{equation*}
\omega_{\ess}(\A)=\underset{t\rightarrow \infty }{\lim }\dfrac{\ln (\|T(t)\|_{\ess})}{t}.
\end{equation*}
\end{definition}

\begin{definition}[Asynchronous Exponential Growth]
\cite[Definition 2.2]{Webb87} \newline \label{Def:AEG}
Let $\{T(t)\}_{t\geq 0}$ be a $C_{0}$-semigroup with infinitesimal generator 
$\A$ in the Banach space $\X$. We say that $\{T(t)\}_{t\geq 0}$ has
asynchronous exponential growth with intrinsic growth constant $\lambda
_{0}\in \mathbb{R}$ if there exists a nonzero finite rank projection $P_{0}$
in $\X$ such that $\lim_{t\rightarrow \infty }e^{-\lambda _{0}t}T(t)=P_{0}$.
\end{definition}
We recall the following standard result (see e.g. \cite{Clement87} Theorem 9.11, p. 224).
\begin{theorem}
\label{Thm:StandardAsync}Let $X$ be a Banach lattice and let $\{T(t)\}_{t\geq 0}$ be a positive $C_{0}$-semigroup on $\X\ $ with
infinitesimal generator $\A$. If $\{T(t)\}_{t\geq 0}$ is irreducible and if 
\begin{equation*}
\omega_{\ess}(\A)<\omega_{0}(\A)
\end{equation*}
then $\{T(t)\}_{t\geq 0}$ has asynchronous exponential growth with intrinsic
growth constant $\lambda _{0}=\omega_{0}(\A)$ and spectral projection $P_0$ of rank one.
\end{theorem}
Now, we need to introduce the following assumption.
\begin{assumption}\label{Assump1}
The integral operator $K$ is weakly compact.
\end{assumption}
\begin{remark}\label{Rem:Weak-Compact}
According to the general criterion of weak compactness (see e.g. Section 4 in \cite{Weis88}), $K$ is weakly compact if and only if
$$\lim_{|E|\to 0}\sup_{y\in [0,m]} \int_E \beta(s,y)ds=0$$
and (additionally for $m=\infty$)
$$\lim_{c\to \infty}\sup_{y\in(0,\infty)}\int_c^\infty \beta(s,y)ds=0.$$
This is satisfied \textit{e.g.} if there exists $\hat{\beta}\in L^1(0,m)$ such that $\beta(s,y)\leq \hat{\beta}(s)$. In particular, this is the case if $m<\infty$ and $\beta$ is continuous on $[0,m]^2$ which occurs in \cite{Farkas2010}.
\end{remark}
We are ready to give the main result of this subsection.

\begin{theorem}\label{Thm:Asynch} 	
Suppose that Assumption \ref{Assump1} holds. The semigroup $\{T_{\A}(t)\}_{t\geq 0}$ has a spectral gap if and only if 
Assumption \eqref{Eq:H1_bis} is satisfied.  Moreover, under the stronger assumptions where \eqref{Eq:H1}-\eqref{Eq:H2}-\eqref{Eq:H3} hold, then the semigroup $\{T_\A(t)\}_{t\geq 0}$ has asynchronous exponential growth.
\end{theorem}

\begin{proof}
The semigroups $\{T_\A(t)\}_{t\geq 0}$ and $
\{T_{A+B_1+B_2}(t)\}_{t\geq 0}$ are related by the \break Duhamel equation
\begin{equation*}
T_\A(t)=T_{A+B_1+B_2}(t)+\int_{0}^{t}T_{A+B_1+B_2}(t-s)B_3T_\A(s)ds.
\end{equation*}
Since $B_3$ is a weakly compact operator then so is $T_{A+B_1+B_2}(t-s)B_3T_\A(s)$ for all $%
s\geq 0.\ $ It follows that the \textit{strong} integral
\begin{equation*}
\int_{0}^{t}T_{A+B_1+B_2}(t-s)B_3T_\A(s)ds
\end{equation*}
is a weakly compact operator (see \cite{Mokhtar2004} Theorem 1 or \cite{Schluchtermann92} Theorem 2.2). Hence $T_\A(t)-T_{A+B_1+B_2}(t)$ is a weakly compact operator
and consequently (see \cite{Mokhtar97} Theorem 2.10, p. 24) $\{T_\A(t)\}_{t\geq 0}$ and $\{T_{A+B_1+B_2}(t)\}_{t\geq 0}$ have the \textit{same}
essential type 
\begin{equation*}
\omega _{\ess}(\A)=\omega _{\ess}(A+B_1+B_2),
\end{equation*}
in particular
\begin{equation*}
\omega _{\ess}(\A)\leq \omega _{0}(A+B_1+B_2).
\end{equation*}%
Note that $s(A+B_1+B_2)=\omega _{0}(A+B_1+B_2)$ and $s(\A)=\omega_{0}(\A)$
since $\{T_\A(t)\}_{t\geq 0}$ and $\{T_{A+B_1+B_2}(t)\}_{t\geq 0}$ are positive semigroups
on $L^{1}$ spaces (see e.g. \cite{EngelNagel2000} Theorem VI.1.15, p. 358). If \eqref{Eq:H1_bis} is satisfied, then applying Theorem \ref{Thm:SpectBound_inf} and Theorem \ref{Thm:SpectBound} we get respectively
$$\omega_0(\A)>-\infty \quad \text{and} \quad \omega_0(A+B_1+B_2)=-\infty$$
so
\begin{equation*}
\omega _{\ess}(\A)<\omega _{0}(\A)
\end{equation*}
whence the existence of a spectral gap. If \eqref{Eq:H1_bis} does not hold, then it is clear by Theorem \ref{Thm:SpectBound} that $\omega_0(\A)=-\infty$ and there cannot be a spectral gap. If the assumptions \eqref{Eq:H1}-\eqref{Eq:H2}-\eqref{Eq:H3} hold, then we immediately see that \eqref{Eq:H1_bis} is satisfied and we have a spectral gap. Combining this with the irreducibility of $\{T_\A(t)\}_{t\geq 0}$ obtained in Theorem \ref{Thm:Irr}, we use 
Theorem \ref{Thm:StandardAsync} to end the proof.
\end{proof}

\subsection{Time asymptotics in absence of irreducibility}

Two kinds of results are given. We start with:

\begin{theorem}\label{Thm:No_Irr}
Suppose that Assumption \ref{Assump1} holds and that \eqref{Eq:H1_bis} is satisfied, i.e. that the $C_0$-semigroup $\{T_{\A}(t)\}_{t\geq 0}$ has a spectral gap. Then, the peripheral spectrum of $\A$ reduces to $s(\A)$, i.e.
$$\sigma(\A)\cap \{\lambda\in \C: \Re(\lambda)=s(\A)\}=\{s(\A)\};$$
and there exists a nonzero finite rank projection $P_0$ in $\X$ such that
\begin{equation}
\label{Eq:AEG_Gap} \lim_{t\to \infty}\|e^{-s(\A)t}T_{\A}(t)-e^{tD}P_0\|_{\X}=0
\end{equation}
where $D:=(s(\A)-\A)P_0$.
\end{theorem}

\begin{proof}
It follows from \cite{Clement87}, Theorem 9.10, p. 223 and Theorem 9.11, p. 224.
\end{proof}

Note that, if $\{T_{\A}(t)\}_{t\geq 0}$ is irreducible, then it has also a spectral gap, whence the asynchronous exponential growth of the semigroup. In this case, the spectral bound $s(\A)$ is algebraically simple (see e.g. \cite{Clement87}, Theorem 9.10, p.223) and the nilpotent operator $D$ that appears in \eqref{Eq:AEG_Gap} is actually zero. Whether the spectral bound could be semi-simple (i.e. a simple pole of the resolvent) when $\{T_{\A}(t)\}_{t\geq 0}$ is not irreducible, is an open problem.

It may happen that $\{T_{\A}(t)\}_{t\geq 0}$ is not irreducible but leaves invariant a subspace on which it is irreducible. This is our second result.

\begin{theorem}\label{Thm:Small_AEG}
Suppose that Assumption \ref{Assump1} holds and that \eqref{Eq:H3} and \eqref{Eq:H1_bis} are satisfied. We thus define
$$b_1:=\inf\{\delta\in[0,m] : \int_0^\delta \int_\delta^m \beta(s,y)dyds>0\}<m.$$
We suppose also that 
\begin{equation}\label{Eq:HypC1}
|\supp(c_1)\cap [b_1,m]| \neq 0
\end{equation}
and
\begin{equation}\label{Eq:HypBeta}
\forall \varepsilon\in (b_1,m): \int_0^\varepsilon \int_\varepsilon^m \beta(s,y)dyds>0
\end{equation}
so we can define
$$b_2:=\inf\{\delta\in [b_1,m] : |\supp(c_1)\cap [b_1,\delta)|\neq 0\}.$$
Let 
$$\Y:=L^1(b_1,m)\times L^1(b_2,m).$$
Then $\Y$ is invariant under $\{T_{\A}(t)\}_{t\geq 0}$, and there exists a projection $\tilde{P_0}$ of rank one, in $\Y$ such that
$$\lim_{t\to \infty}e^{-s(\A_{\Y}) t}T_{\A_{\Y}}(t)u=\tilde{P_0} u$$
for every $u\in \Y$, where 
$$\{T_{\A_{\Y}}(t)\}_{t\geq 0}={\{T_{\A}(t)\}_{t\geq 0}}_{\mid \Y}$$
and $\A_{\Y}$ is the generator of $\{T_{\A_{\Y}}(t)\}_{t\geq 0}$.
\end{theorem}

\begin{proof}\mbox{}
Define the operator
\begin{equation*}
\begin{array}{rcl}
\A_{\Y}\begin{pmatrix}
u_1\\
u_2
\end{pmatrix}=\begin{pmatrix}
-(\overline{\gamma_1} u_1)' \\
-(\underline{\gamma_2} u_2)'
\end{pmatrix}+\begin{pmatrix}
-(\overline{\mu}+\overline{c_1})u_1+\overline{c_2}u_2+\int_{b_1}^{m} \overline{\beta}(\cdot,y)u_1(y)dy) \\
-\underline{c_2} u_2+\underline{c_1} u_1
\end{pmatrix},
\end{array}
\end{equation*}
with domain
$$D(\A_{\Y})=\{(u_1,u_2)\in W^{1,1}(b_1,m)\times W^{1,1}(b_2,m): u_1(b_1)=0, u_2(b_2)=0\},$$
where
$$\overline{\gamma_1}={\gamma_1}_{\mid [b_1,m]}, \ \overline{\mu}=\mu_{\mid [b_1,m]}, \ \overline{c_1}={c_1}_{\mid [b_1,m]}, \ \overline{c_2}={c_2}_{\mid [b_1,m]}, \ \overline{\beta}={\beta}_{\mid [b_1,m]\times [b_1,m]}, $$
and 
$$\underline{\gamma_2}={\gamma_2}_{\mid [b_2,m]}, \ \underline{c_1}={c_1}_{\mid [b_2,m]}, \ \underline{c_2}={c_2}_{\mid [b_2,m]}.$$
As in Theorem \ref{Thm:Generation}, $\A_{\Y}$ generates a $C_0$-semigroup $\{T_{\A_{\Y}}(t)\}_{t\geq 0}$. Using the point 3.(a) of the proof of Theorem \ref{Thm:Irr}, with $\varepsilon=b_1$, we know that
$$L^1(b_1,m)\times L^1(b_1,m)$$
is a closed ideal of $\X$ that is invariant under $(\lambda I-\A)^{-1}$ for every $\lambda>s(\A)$. Then, using the point 3.(b) of the proof of Theorem \ref{Thm:Irr}, with $k=b_2$, we can prove that $\Y$ is a closed ideal of $\X$ that is invariant under $(\lambda I-\A)^{-1}$ for every $\lambda>s(\A)$. Consequently 
$${\{T_{\A}(t)\}_{t \geq 0}}_{\mid \Y}=\{T_{\A_{\Y}}(t)\}_{t \geq 0}.$$
By means of \eqref{Eq:HypBeta} and by definition of $b_1$, we see that
$$\forall \varepsilon \in(b_1,m) : \int_{b_1}^\varepsilon \int_{\varepsilon}^m \beta(s,y)dyds>0.$$
Using \eqref{Eq:HypC1} and by definition of $b_2$, we have
$$\inf \supp(\overline{c_1})=\inf \supp(\underline{c_1})=b_2.$$
Consequently, as for Theorem \ref{Thm:Irr}, $\A_{\Y}$ is irreducible and
$$\omega_\ess(\A_{\Y})<\omega_0(\A_{\Y}).$$
Therefore, as in Theorem \ref{Thm:Asynch}, the semigroup $\{T_{\A_{\Y}}(t)\}_{t\geq 0}$ has the property of asynchronous exponential growth. Thus we get
$$\lim_{t\to \infty} e^{-s(\A_{\Y})t}T_{\A_{\Y}}(t)=\tilde{P_0},$$
where $\tilde{P_0}$ is a projection of rank one in $\Y$.
\end{proof}

Note that $s(\A_{\Y})\leq s(\A)$. It is unclear whether the inequality is strict.

\section{Models with unbounded sizes}

In this section we consider the following model

\begin{equation}\label{Eq:Model_inf}
\left\{
\begin{array}{rcl}
\partial_t u_1(t,s)+\partial_s(\gamma_1(s)u_1(t,s))&=&-\mu(s)u_1(t,s)+\int_0^\infty \beta(s,y)u_1(t,y)dy \\
&& -c_1(s)u_1(t,s)+c_2(s)u_2(t,s),\\
\partial_t u_2(t,s)+\partial_s(\gamma_2(s)u_2(t,s))&=&c_1(s)u_1(t,s)-c_2(s)u_2(t,s),
\end{array}
\right.
\end{equation}
for $s, t\geq 0,$ with the Dirichlet boundary conditions \eqref{Eq:BoundCond}.
Let the Banach space
$$\X=(L^1(0,\infty)\times L^1(0,\infty),\|\cdot \|_{\X})$$
with norm
$$\|(x_1,x_2)\|_{\X}=\|x_1\|_{L^1(0,\infty)}+\|x_2\|_{L^1(0,\infty)}.$$
We denote by $\X_+$ the nonnegative cone of $\X$. We now suppose in all this section the following hypotheses on the different parameters:
\begin{enumerate}
\item $\mu, c_1, c_2 \in L^{\infty}(0,\infty), \gamma_1, \gamma_2 \in W^{1,\infty}(0,\infty),$
\item $\beta, \mu, c_1, c_2 \geq 0$ and there exists $\gamma_0>0$ such $\gamma_1(s)\geq \gamma_0, \gamma_2(s)\geq \gamma_0$ a.e. $s\geq 0$,
\item the operator 
$$K: L^1(0,\infty)\ni u\mapsto  \int_0^{\infty} \beta(\cdot,y)u(y)dy \in L^1(0,\infty)$$
is bounded (see Remark \ref{Rem:Bounded}).
\end{enumerate}

Using \eqref{Eq:Model_inf}, we define
\begin{equation*}
\begin{array}{rcl}
\A \begin{pmatrix}
u_1\\
u_2
\end{pmatrix}&=&A \begin{pmatrix}
u_1 \\
u_2
\end{pmatrix}+B \begin{pmatrix}
u_1 \\
u_2
\end{pmatrix}\\
&=&\begin{pmatrix}
-(\gamma_1 u_1)' \\
-(\gamma_2 u_2)'
\end{pmatrix}+\begin{pmatrix}
-(\mu+c_1)u_1 +c_2u_2+\int_0^\infty \beta(\cdot,y)u_1(y)dy\\
-c_2 u_2+c_1 u_1
\end{pmatrix},
\end{array}
\end{equation*}
with domain
$$D(A)=\{(u_1,u_2)\in W^{1,1}(0,\infty)\times W^{1,1}(0,\infty):u_1(0)=0, u_2(0)=0\}.$$
We decompose $B$ into three operators:
\begin{equation*}
\begin{array}{rcl}
B\begin{pmatrix}
u_ 1\\
u_2
\end{pmatrix}
&=&B_1\begin{pmatrix}
u_1 \\
u_2
\end{pmatrix}
+B_2\begin{pmatrix}
u_1 \\
u_2
\end{pmatrix}
+B_3\begin{pmatrix}
u_1 \\
u_2
\end{pmatrix}\\
&=&\begin{pmatrix}
-(\mu+c_1)u_1 \\
-c_2 u_2
\end{pmatrix} 
+\begin{pmatrix}
c_2 u_2 \\
c_1 u_1
\end{pmatrix}
+\begin{pmatrix}
\int_0^\infty \beta(\cdot,y)u_1(y)dy \\
0
\end{pmatrix}.
\end{array}
\end{equation*}
We are then concerned with the following Cauchy problem
\begin{equation*}
\left\{
\begin{array}{rcl}
U'(t)&=&\A U(t), \\
U(0)&=&(u^0_1,u^0_2)\in \X,
\end{array}
\right.
\end{equation*}
where
$$U(t)=(u_1(t),u_2(t))^T.$$

\subsection{Semigroup generation}

\begin{lemma}\label{Lemma:Resolv_inf}
Let $H:=(h_1,h_2)\in \X$ and $\lambda \in \R$. The solution of
\begin{equation}\label{Eq:System}
\left\{
\begin{array}{rcl}
\lambda u_1+(\gamma_1 u_1)'&=&h_1, \\
\lambda u_2+(\gamma_2 u_2)'&=&h_2, \\
u_1(0)&=&u_2(0)=0,
\end{array}
\right.
\end{equation}
is given by
\begin{equation}\label{Eq:Range_inf}
\left\{
\begin{array}{rcl}
u_1(s)&=& \displaystyle \dfrac{1}{\gamma_1(s)} \int_0^s h_1(y) \exp\left(-\int_y^s \dfrac{\lambda}{\gamma_1(z)}dz\right)dy \vspace{0.1cm}, \\
u_2(s)&=& \displaystyle \dfrac{1}{\gamma_2(s)} \int_0^s h_2(y) \exp\left(-\int_y^s \dfrac{\lambda}{\gamma_2(z)}dz\right)dy,
\end{array}
\right.
\end{equation}
for every $s\geq 0$. In particular, 
$U:=(u_1, u_2)\in D(A)$ if and only if $U\in \X$. Moreover, if $H\in \X_+$, then
$$\supp u_1=[\inf \supp(h_1),\infty), \qquad \supp u_2=[\inf \supp(h_2),\infty).$$
\end{lemma}

\begin{remark}\label{Rem:Abbreviations}
In all the sequel, for the simplicity of notations, we write symbolically $(\lambda-A)U=H$ instead of \eqref{Eq:System} even if $U$ need not belong to the domain of $A$. We will also use similar symbolic abbreviations in similar contexts.
\end{remark}

\begin{theorem}\label{Thm:Generation_inf}
The operator $\A$ generates a $C_0$-semigroup $\{T_{\A}(t)\}_{t\geq 0}$ of bounded linear operators on $\X$.
\end{theorem}

\begin{proof}
As in the finite case, we only need to prove that $A$ generates a contraction $C_0$-semigroup. The fact that $D(A)$ is densely defined in $\X$ is clear. As before, the range condition
$$(\lambda I-A)U=H,$$
where $U=(u_1,u_2)$ and $H=(h_1,h_2)\in \X$, is verified for every $\lambda>s(A)$.

It remains to prove that $A$ is a dissipative operator. Let $\lambda>0$, $U=(u_1,u_2)\in D(A)$ and $H:=(h_1,h_2)=(\lambda I-A)U$. We want to prove that
$$\|h_i\|_{L^1(0,\infty)}\geq \lambda \|u_i\|_{L^1(0,\infty)}, \quad \forall i\in\{1,2\}.$$
Let $i\in\{1,2\}$. We know that $u_i(0)=0$ and
$$\lambda u_i(s)+(\gamma_i u_i)'(s)=h_i(s), \quad \forall s\in (0,\infty).$$
An integration then leads to
$$\lambda \|u_i\|_{L^1(0,\infty)}+\int_0^\infty (\gamma_i u_i)'(s)\sign(u_i(s))ds=\int_0^\infty h_i(s)\sign(u_i(s))ds.$$
Since $u_i\in W^{1,1}(0,\infty)\hookrightarrow C([0,\infty))$, we get
\begin{eqnarray*}
&&\int_{0}^{m}(\gamma_i u_i)^{\prime }sign(u_i(s))ds =\gamma_i (m)\left\vert u_i(m)\right\vert,
\end{eqnarray*}
for every finite $m>0$. Hence
$$\int_0^\infty (\gamma_i u_i)'\sign(u_i(s))ds=\lim_{m\to \infty}\int_0^m (\gamma_i u_i)'\sign(u_i(s))ds=0$$
and we have
$$ \lambda \|u_i\|_{L^1}=\int_0^\infty h_i(s)\sign(u_i(s))ds\leq \|h_i\|_{L^1}$$ so the dissipativity of $A$ follows. Finally, $A$ generates a contraction $C_0$-semigroup $\{T_{A}(t)\}_{t\geq 0}$ by Lumer-Phillips Theorem and the operators $A+B_1$, $A+B_1+B_2$, $\A$ also generate a quasi-contraction $C_0$-semigroup $\{T_{A+B_1}(t)\}_{t\geq 0}$, $\{T_{A+B_1+B_2}(t)\}_{t\geq 0}$ and $\{T_{\A}(t)\}_{t\geq 0}$ respectively, since $B_1, B_2$ and $B_3$ are bounded operators.
\end{proof}

\subsection{On irreducibility}

Define the following hypotheses:
\begin{equation}\label{Eq:H4}
\forall \varepsilon \in (0,\infty): \quad \int_0^\varepsilon \int_\varepsilon^\infty \beta(s,y)dyds>0,
\end{equation}
\begin{equation}\label{Eq:H5}
\inf \supp c_1=0,
\end{equation}
\begin{equation}\label{Eq:H6}
\sup \supp c_2=\infty.
\end{equation}

\begin{theorem}\label{Thm:Irr_inf}
The $C_0$-semigroup $\{T_{\A}(t)\}_{t\geq 0}$ is irreducible if and only if the assumptions \eqref{Eq:H4}-\eqref{Eq:H5}-\eqref{Eq:H6} are satisfied.
\end{theorem}

\begin{proof}
The proof is similar to that of Theorem \ref{Thm:Irr}.
\end{proof}

\subsection{Asynchronous exponential growth}\label{Sec:Async}

We now introduce the following assumption:
\begin{assumption}\label{Assump2}
The integral operator $K$ is weakly compact. 
\end{assumption}
\noindent (see Remark \ref{Rem:Weak-Compact}). In contrast to the finite case, the asynchronous exponential growth needs an additional condition.

\begin{theorem}\label{Thm:AEG_inf}
Suppose that Assumption \ref{Assump2} holds and let the operator
\begin{equation}\label{Eq:Op_B} 
\B:=A+B_1+B_2
\end{equation}
with domain $D(\B)=D(A)$ a. If
\begin{equation}\label{Cond:spectral_bound}
s(\A)>s(\B)
\end{equation}
holds, then the semigroup $\{T_{\A}\}_{t\geq 0}$ has a spectral gap. In addition to \eqref{Cond:spectral_bound}, if  \eqref{Eq:H4}-\eqref{Eq:H5}-\eqref{Eq:H6} are also satisfied, then the semigroup $\{T_{\A}(t)\}_{t\geq 0}$ has asynchronous exponential growth.
\end{theorem}

\begin{proof}
As in the finite case, the weak compactness of $B_3$
implies that $\{T_{\A}(t)\}_{t\geq 0}$ and $\{T_{\B}(t)\}_{t\geq 0} $ have the same essential spectrum, and consequently the same essential type:
\begin{equation*}
\omega _{\ess}(\A)=\omega _{\ess}\left(\B\right).
\end{equation*}
Since
\begin{equation*}
\omega_{\ess}\left(\B\right)\leq s\left(\B\right)
\end{equation*}
then, using the assumption \eqref{Cond:spectral_bound}, we obtain
\begin{equation*}
\omega _{\ess}\left(\A\right)\leq s\left(\B\right)<s\left(\A \right).
\end{equation*}
Thus $\{T_{\A}(t)\}_{t\geq 0}$ exhibits a spectral gap in this case. Finally, the assumptions \eqref{Eq:H4}-\eqref{Eq:H5}-\eqref{Eq:H6}
ensure with Theorem \ref{Thm:Irr_inf}, that the semigroup $\{T_{\A}(t)\}_{t\geq 0}$ is irreducible, and therefore the asynchronous behavior is proved.
\end{proof}

\begin{remark}\label{Rem:Spec_gap-necessary}
One can show (see Remark \ref{Rem:ess-type}) that the condition \eqref{Cond:spectral_bound} is necessary for $\{T_{\A}(t)\}_{t\geq 0}$ to have a spectral gap.
\end{remark}

By means of \eqref{Eq:Voigt}, we can see that \eqref{Cond:spectral_bound} is satisfied if and only if 
\begin{equation*}
\lim_{\lambda \rightarrow s(\mathcal{B})}r_{\sigma }\left( B_{3}\left(
\lambda -\mathcal{B})^{-1}\right) \right) >1
\end{equation*}
holds, see Lemma \ref{Lemma:Voigt}.

\subsection{Further spectral results}

The object of this subsection is to show that the real spectrum of the differential operators appearing in $\B$ is connected and to estimate their spectral bounds. These results will be used in Subsection \ref{Sec:SpectralGap} to show, in some situations of practical interest, the existence or the absence of a spectral gap for $\{T_{\A}(t)\}_{t\geq 0}$. 

\subsubsection{Spectral theory of uncoupled systems}

Define the operators
$$A^i_0 u=-(\gamma_i u)', \quad \forall i\in \{1,2\}$$
for every $u\in D(A^i_0)=\{u\in W^{1,1}(0,\infty): u(0)=0\}$, $i\in\{1,2\}$, so that
$$A=\begin{pmatrix}
A^1_0 & 0 \\
0 & A^2_0
\end{pmatrix}.$$

\begin{theorem}\label{Prop:Spec_B}
We have
$$\sigma(A)\cap \R=(-\infty,0].$$
In particular, $s(A)=0$.
\end{theorem}

\begin{proof}
Note that $A$ generates a contraction $C_0$-semigroup, so
$$\sigma(A)\subset \{\lambda\in \C: \Re(\lambda)\leq 0\}.$$
Let $\lambda\in \R$ and $H:=(h_1,h_2)\in \X_+$. The solution $U_\lambda$ of
$$(\lambda I-A)U_\lambda=H, \qquad U_\lambda(0)=(0,0)$$
(see Remark \ref{Rem:Abbreviations}) given by \eqref{Eq:Range_inf} is nonincreasing in $\lambda$. Consequently
$$U_\lambda \not\in \X \Rightarrow U_\alpha \notin \X \ \forall \alpha \leq \lambda$$
and
$$\sigma(A)\cap \R=(-\infty,s(A)].$$
Let $\lambda=0$, $i\in \{1,2\}$ and $h\in L^1_+(0,\infty)$. Suppose that $\lambda \in \rho(A^i_0)$. Then $u:=(\lambda-A^i_0)^{-1}h$ is given by
$$u(s)=\dfrac{1}{\gamma_i(s)}\int_0^s h(y)dy\geq 0, \quad \forall s\in[0,m].$$
So we get
\begin{eqnarray*}
\int_0^\infty u(s) ds&=&\int_0^\infty \dfrac{1}{\gamma_i(s)}\int_0^s h(y)dyds=\int_0^\infty h(y) \int_y^\infty \dfrac{1}{\gamma_i(s)}dsdy \\
&\geq& \dfrac{1}{\|\gamma_i\|_{L^\infty}}\int_0^\infty h(y)\int_y^\infty dsdy=\infty.
\end{eqnarray*}
Thus $u \notin L^1(0,\infty)$ and $0\in \sigma(A^i_0)$. Consequently
$$s(A)=\max\{s(A^1_0), s(A^2_0)\}=0.$$
\end{proof}
Now, define the operators
$$A^1_\mu u=-(\gamma_1 u)'-\mu u, \quad A^2_{c_2} u=-(\gamma_2 u)'-c_2 u,$$
for every $u\in D(A^1_{\mu})=D(A^2_{c_2})=\{u\in W^{1,1}(0,\infty): u(0)=0\}$. Since $\mu \geq 0$ then $s(A^1_{\mu})\leq 0.$ We give now more
information on the spectrum of $A^1_{\mu}.$ 

\begin{theorem}\label{Prop:Spec_Bmu}
We have
$$\left(-\infty,-\limsup_{x\to \infty}\mu(x)\right]\subset\sigma\left(A^1_{\mu}\right)$$
and
$$-\liminf_{x\to \infty}\mu(x) \geq s(A^1_{\mu})\geq -\limsup_{x\to \infty} \mu(x).$$
In particular
$$s\left(A^1_{\mu}\right)=\lim_{x\to \infty}\mu(x)$$
if the latter exists.
\end{theorem}

\begin{proof}
Let $\lambda \in \R$ and $h\in L^1(0,\infty)$. The solution of
$$\left(\lambda I-A^1_{\mu}\right)u=h, \qquad u(0)=0$$
(see Remark \ref{Rem:Abbreviations} for the abbreviation) is given by
\begin{equation}\label{Eq:Resolv_Amu}
u(s):=\dfrac{1}{\gamma_1(s)}\displaystyle \int_0^s h(y)\exp\left(-\int_y^s \dfrac{\lambda+\mu(z)}{\gamma_1(z)}dz\right)dy
\end{equation}
that is nonincreasing in $\lambda$, consequently
$$\sigma\left(A^1_{\mu}\right)\cap \R=\left(-\infty,s\left(A^1_{\mu}\right)\right].$$
Now, let $\varepsilon>0$ ($\varepsilon$ need not be small), $h\in L^1(0,\infty)$ and
$$\lambda:=-\liminf_{x\to \infty}\mu(x)+\varepsilon.$$
The solution of
$$\left(\lambda I-A^1_{\mu}\right)u=h, \qquad u(0)=0,$$
is given by \eqref{Eq:Resolv_Amu}. Then
\begin{eqnarray*}
&&\displaystyle \int_0^\infty |u(s)|ds \\
&\leq& \dfrac{1}{\gamma_0}\int_0^\infty |h(y)|\int_y^\infty \exp\left(-\int_y^s \dfrac{-\liminf_{x\to \infty}\mu(x)+\varepsilon+\mu(z)}{\gamma_1(z)}dz\right)dsdy.
\end{eqnarray*}
We know that there exists $\eta>0$ such that for every $y\geq \eta$ we have $\mu(y)\geq \liminf_{x\to \infty}\mu(x) -\varepsilon/2$. So we get first
\begin{equation*}
\begin{array}{rcl}
&& \displaystyle \int_{\eta}^\infty \dfrac{|h(y)|}{\gamma_0}\int_y^\infty \exp\left(-\int_y^s \dfrac{-\liminf_{x\to \infty}\mu(x) +\varepsilon+\mu(z)}{\gamma_1(z)}dz\right)dsdy \vspace{0.1cm} \\
&\leq& \displaystyle \int_{\eta}^\infty \dfrac{|h(y)|}{\gamma_0}\int_y^\infty \exp\left(-\int_y^s \dfrac{\varepsilon/2}{\|\gamma_1\|_{L^\infty}}\right)dsdy \vspace{0.1cm} \\
&\leq& \displaystyle \int_{\eta}^\infty \dfrac{|h(y)|}{\gamma_0}\int_y^\infty \exp\left(-\dfrac{\varepsilon(s-y)}{2\|\gamma_1\|_{L^\infty}}\right)dsdy \vspace{0.1cm} \\
&\leq& \displaystyle \dfrac{2 \|\gamma_1\|_{L^\infty} }{\varepsilon \gamma_0}\int_{\eta}^\infty |h(y)|dy<\infty.
\end{array}
\end{equation*}
Moreover, for every $y\in[0,\eta]$, we have
\begin{flalign*}
& \qquad \displaystyle\int_y^\infty \exp\left(-\int_y^s \dfrac{-\liminf_{x\to \infty}\mu(x) +\varepsilon+\mu(z)}{\gamma_1(z)}dz\right)ds \\
&\leq C_1\int_y^\infty \exp\left(-\int_0^s\dfrac{-\liminf_{x\to \infty}\mu(x) +\varepsilon+\mu(z)}{\gamma_1(z)}dz\right)ds\\
&\leq C_2\int_y^\infty \exp\left(-\int_0^s\dfrac{-\liminf_{x\to \infty}\mu(x) +\varepsilon+\mu(z)}{\gamma_1(z)}dz\right)ds,
\end{flalign*}
where
$$C_1:=\displaystyle \exp\left(\int_0^y \dfrac{|-\liminf_{x\to \infty}\mu(x) +\varepsilon+\mu(z)|}{\gamma_1(z)}dz\right)$$
and
$$C_2:=\displaystyle \exp\left(\dfrac{\eta(|\varepsilon-\liminf_{x\to \infty}\mu(x)|+\|\mu\|_{L^\infty})}{\gamma_0}\right)<\infty.$$
Note that, for every $y\in[0,\eta]$
\begin{equation*}
\begin{array}{rcl}
&&\displaystyle \int_y^\infty \exp\left(-\int_0^s\dfrac{-\liminf_{x\to \infty}\mu(x) +\varepsilon+\mu(z)}{\gamma_1(z)}dz\right)ds \\
&\leq&\displaystyle \int_{\eta}^\infty \exp\left(-\int_0^{\eta}\dfrac{-\liminf_{x\to \infty}\mu(x) +\varepsilon+\mu(z)}{\gamma_1(z)}dz\right)\exp\left(-\int_{\eta}^s\dfrac{\varepsilon/2}{\gamma_1(z)}dz\right)ds \\
&&+\displaystyle \int_0^\eta \exp\left(-\int_0^s \dfrac{-\liminf_{x\to \infty}\mu(x) +\varepsilon+\mu(z)}{\gamma_1(z)}dz\right)ds.
\end{array}
\end{equation*}
Consequently
$$\int_{0}^\eta \dfrac{|h(y)|}{\gamma_0}\int_y^\infty \exp\left(-\int_y^s \dfrac{-\liminf_{x\to \infty}\mu(x) +\varepsilon+\mu(z)}{\gamma_1(z)}dz\right)dsdy < \infty$$
and
$$\displaystyle \int_0^\infty |u(s)|ds<\infty$$
so $u\in L^1(0,\infty)$ and
$$-\liminf_{x\to \infty}\mu(x)+\varepsilon \in \rho(A^1_\mu)$$
for every $\varepsilon>0$ whence
$$s(A^1_\mu)\leq -\liminf_{x\to \infty}\mu(x).$$
Now let $\varepsilon>0$,  $h\in L^1_+(0,\infty)$ and
$$\lambda:=-\limsup_{x\to \infty}\mu(x) -\varepsilon.$$
Suppose that $\lambda\in \rho(A^1_{\mu})$, then $u:=(\lambda-A^1_\mu)^{-1}h$ is given by \eqref{Eq:Resolv_Amu}. We know that there exists $\overline{y}>0$ and $\overline{s}>\overline{y}$ such that
$$\int_0^{\overline{y}}h(z)dz>0$$
and
$$\mu(s)\leq \limsup_{x\to \infty}\mu(x)+\varepsilon/2$$
for every $s\geq \overline{s}$. Consequently we get
\begin{flalign*}
& \qquad \displaystyle \int_0^\infty u(s)ds\\
&= \displaystyle \int_0^{\infty} h(y)\int_{y}^\infty \dfrac{1}{\gamma_1(s)} \exp\left(-\int_{y}^s \left(\dfrac{\mu(z)-(\limsup_{x\to \infty}\mu(x)+\varepsilon)}{\gamma_1(z)}\right)dz\right)dsdy\\
&\geq \displaystyle \int_0^{\overline{y}} h(y)\int_{\overline{s}}^\infty \left[\dfrac{1}{\gamma_1(s)} \exp\left(-\int_{\overline{s}}^s \left(\dfrac{\mu(z)-(\limsup_{x\to \infty}\mu(x) +\varepsilon)}{\gamma_1(z)}\right)dz\right) \right.\\
&\quad \left.\exp\left(-\int_{y}^{\overline{s}} \left(\dfrac{\mu(z)-(\limsup_{x\to \infty}\mu(x) +\varepsilon)}{\gamma_1(z)}\right)dz\right) \right]ds dy \\
&\geq \displaystyle \int_0^{\overline{y}} \dfrac{h(y)}{\|\gamma_1\|_{L^\infty}}\int_{\overline{s}}^\infty  \left[\exp\left((y-\overline{s})\left(\dfrac{\|\mu\|_{L^\infty}+\limsup_{x\to \infty}\mu(x)+\varepsilon}{\gamma_0}\right)\right)\right. \\
&\quad \left. \exp\left(\dfrac{\varepsilon (s-\overline{s})}{2\|\gamma_1\|_{L^\infty}}\right)\right] ds dy
\end{flalign*}

\begin{flalign*}
&\geq \displaystyle \int_0^{\overline{y}} \dfrac{h(y)}{\|\gamma_1\|_{L^\infty}}dy\int_{\overline{s}}^\infty \left[\exp\left(-\overline{s}\left(\dfrac{\|\mu\|_{L^\infty}+\limsup_{x\to \infty}\mu(x)+\varepsilon}{\gamma_0}\right) \right) \right. \\
&\quad \left. \exp\left(\dfrac{\varepsilon (s-\overline{s})}{2\|\gamma_1\|_{L^\infty}}\right)\right]ds \\
&=\infty
\end{flalign*}
so $u\notin L^1(0,\infty)$ and
$$-\limsup_{x\to\infty}\mu(x)-\varepsilon\in \sigma(A^1_\mu)$$
for every $\varepsilon>0$ whence
$$s(A^1_\mu)\geq -\limsup_{x\to \infty}\mu(x).$$
\end{proof}

\begin{remark}\label{Rem:Estimates}
Note that similar estimates hold for $A^2_{c_2}$.
\end{remark}

\subsubsection{Spectral theory of coupled systems}

Define the operator
$$A^1_{\mu+c_1} u=-(\gamma_1 u)'-(\mu+c_1)u, \qquad $$
with $D(A^1_{\mu+c_1})=\{u\in W^{1,1}(0,\infty): u(0)=0\}$. Let $H:=(h_1,h_2)\in \X$ and $\lambda\in \R$. The system 
\begin{equation}\label{Eq:System_Coupled}
\left\{
\begin{array}{rcl}
\lambda u_1+(\gamma_1 u_1)'+(\mu+c_1)u_1 -c_2 u_2&=&h_1, \\
\lambda u_2+(\gamma_2 u_2)'+(c_2)u_2-c_1 u_1&=&h_2, \\
u_1(0)&=&u_2(0)=0,
\end{array}
\right.
\end{equation}
can be globally solved by iterations, since it is a perturbed linear Cauchy problem, by writing
\begin{equation*}
\left\{
\begin{array}{rcl}
\lambda u_1+(\gamma_1 u_1)'+(\mu+c_1)u_1&=&c_2 u_2+h_1, \\
\lambda u_2+(\gamma_2 u_2)'+(c_2)u_2&=&c_1 u_1+h_2, \\
u_1(0)&=&u_2(0)=0.
\end{array}
\right.
\end{equation*}
Since $B_2$ is a positive operator then, once $H\in \X_+$, the iterative sequence
\begin{equation*}
\left\{
\begin{array}{rcl}
\lambda u_1^{n+1}+(\gamma_1 u_1^{n+1})'+(\mu+c_1)u_1^{n+1}&=&c_2 u_2^n+h_1, \\
\lambda u_2^{n+1}+(\gamma_2 u_2^{n+1})'+(c_2)u_2^{n+1}&=&c_1 u_1^n+h_2, \\
u_1^{n+1}(0)&=&u_2^{n+1}(0)=0.
\end{array}
\right.
\end{equation*}
(with $u_1^0=u_2^0=0$) is \textit{nonnegative} and then so is its limit. In addition
\begin{equation*}
\begin{pmatrix}
u_1^{n+1} \\
u_2^{n+1}
\end{pmatrix}(s)=\begin{pmatrix}
\dfrac{1}{\gamma_1(s)}\displaystyle \int_0^s \left[h_1(y)+c_2(y)u_2^n(y)\right]e^{-\int_y^s\left(\frac{\lambda+\mu(z)+c_1(z)}{\gamma_1(z)}\right)dz}dy \\
\dfrac{1}{\gamma_2(s)}\displaystyle \int_0^s \left[h_2(y)+c_1(y)u_1^n(y)\right]e^{-\int_y^s\left(\frac{\lambda+c_2(z)}{\gamma_2(z)}\right)dz}dy
\end{pmatrix} \qquad \forall s\geq 0
\end{equation*}
shows by induction that the sequences $u_1^n$ and $u_2^n$ are nonincreasing in $\lambda$. In all the following, we will write symbolically $(\lambda -\B)U=H$ instead of \eqref{Eq:System_Coupled}, even if $U\not\in D(\B)$. Finally, the solution of \eqref{Eq:System_Coupled} always satisfies the Duhamel equation
\begin{equation}\label{Eq:Resolv_Couple}
\begin{pmatrix}
u_1 \\
u_2
\end{pmatrix}(s)=\begin{pmatrix}
\dfrac{1}{\gamma_1(s)}\displaystyle \int_0^s \left[h_1(y)+c_2(y)u_2(y)\right]e^{-\int_y^s\left(\frac{\lambda+\mu(z)+c_1(z)}{\gamma_1(z)}\right)dz}dy \\
\dfrac{1}{\gamma_2(s)}\displaystyle \int_0^s \left[h_2(y)+c_1(y)u_1(y)\right]e^{-\int_y^s\left(\frac{\lambda+c_2(z)}{\gamma_2(z)}\right)dz}dy
\end{pmatrix} \qquad \forall s\geq 0
\end{equation}
and is \textit{nonincreasing} in $\lambda$. Thus, if $\alpha<\lambda$ then $U_{\lambda}\not\in \X \Rightarrow U_{\alpha}\not\in \X$, so
$$\sigma(\B)\cap \R=(-\infty,s(\B)].$$

\begin{remark}\label{Rem:ess-type}
Note that $\omega_{\ess}(\B)=\omega_0(\B)$. Indeed, if $\omega_{\ess}(\B)<\omega_0(\B)$ then for any $\alpha$ such that $\omega_{\ess}(\B)<\alpha<\omega_0(\B)=s(\B)$, the set $\sigma(\B)\cap\{\lambda: \Re(\lambda)\geq \alpha\}$ should consist of a finite set of eigenvalues and this contradicts the fact that $\sigma(\B)\cap \R=(-\infty,s(\B)]$. It follows from the proof of Theorem \ref{Thm:AEG_inf} that $\{T_{\A}(t)\}_{t\geq 0}$ has a spectral gap (i.e. $\omega_{\ess}(\A)<\omega_0(\A)$) if and only if $s(\A)>s(\B)$.
\end{remark}

\begin{theorem}\label{Thm:Spec_Dtilde}
We have
$$-\limsup_{x\to \infty}\mu(x)\leq s(\B)\leq 0$$
and in particular
$$\left(-\infty,-\limsup_{x\to \infty}\mu(x)\right]\subset \sigma(\B).$$
Moreover, if $\liminf_{x\to \infty}\mu(x)>0$ and $\liminf_{x\to \infty}c_2(x)>0$ then
$$s(\B)<0.$$
\end{theorem}

\begin{proof}
Let $\lambda>0$, $H:=(h_1,h_2)\in L^1(0,\infty)\times L^1(0,\infty)$. The solution $U:=(u_1,u_2)$ of
$$(\lambda I-\B)U=H, \quad U(0)=(0,0)$$
is given by \eqref{Eq:Resolv_Couple} and satisfies
\begin{equation}\label{Eq:Sys_SpecDtilde}
\left\{
\begin{array}{rcl}
(\gamma_1 u_1)'+(\lambda+c_1+\mu)u_1-c_2 u_2&=&h_1, \\
(\gamma_2 u_2)'+(\lambda+c_2)u_2-c_1u_1&=&h_2.
\end{array}
\right.
\end{equation}
By adding, we get
\begin{equation}\label{Eq:Sum}(\gamma_1u_1)'+(\gamma_2u_2)'+\lambda(u_1+u_2)+\mu u_1=h_1+h_2=:h.
\end{equation}
We know that the resolvent of $\B$ is a positive operator, so it suffices to take $(h_1,h_2)\in \X_+$. Then $u_1$ and $u_2$ are nonnegative functions and an integration of the latter equation leads to
$$\gamma_1(m)u_1(m)+\gamma_2(m)u_2(m)+\lambda \int_0^m (u_1(s)+u_2(s))ds+\int_0^m \mu(s)u_1(s)ds=\int_0^m h(s)ds$$
for every $m>0$. Consequently
$$\lambda \int_0^m (u_1(s)+u_2(s))ds\leq \int_0^m h(s)ds$$
and
$$\lambda \int_0^\infty (u_1(s)+u_2(s))ds\leq \|h\|_{L^1}<\infty$$
by passing to the limit, whence
$$u_1+u_2\in L^1(0,\infty)$$ 
so $u_1\in L^1(0,\infty)$ and $u_2\in L^1(0,\infty)$. Thus $\lambda \in \rho(\B)$ for every $\lambda>0$ and 
$$s(\B)\leq 0.$$
Now let $H:=(h_1,h_2)\in \X_+$ and $\lambda:=-\limsup_{x\to \infty}\mu(x)-\varepsilon$, with $\varepsilon>0$. We know that there exists $\eta>0$ such that
$$\mu(x)\leq \limsup_{x\to \infty}\mu(x)+\varepsilon/2, \quad \forall x\geq \eta,$$
so
$$\lambda+\mu(x)\leq -\varepsilon/2<0, \quad \forall x\geq \eta.$$
Suppose that $\lambda \in \rho(\B)$, then an integration of \eqref{Eq:Sum} between $\eta$ and $\infty$ implies that
$$0\geq -\gamma_1(\eta)u_1(\eta)-\gamma_2(\eta)u_2(\eta)+\int_{\eta}^\infty (\lambda+\mu(s))(u_1(s)+u_2(s))ds\geq \int_{\eta}^\infty h(s)ds.$$
Taking $h\in L^1(0,\infty)$ such that $\int_{\eta}^\infty h(s)ds>0$ would lead to a contradiction. Thus
$$-\limsup_{x\to \infty}\mu(x)-\varepsilon \in \sigma(\B)$$
for every $\varepsilon>0$ and 
$$s(\B)\geq -\limsup_{x\to \infty}\mu(x).$$
Finally, suppose that $\liminf_{x\to \infty}\mu(x)>0$ and $\liminf_{x\to \infty}c_2(x)>0 $. Let $\varepsilon>0$, then there exists $\eta>0$ such that
$$\mu(x)\geq \varepsilon/2, \quad \forall x\geq \eta.$$ Let $\lambda=0$ and $H:=(h_1,h_2)\in \X_+$. The solution of
$$(\lambda I-\B)U=H, \quad U(0)=(0,0)$$
satisfies \eqref{Eq:Sum} and an integration lead to
$$\gamma_1(m)u_1(m)+\gamma_2(m)u_2(m)+\int_0^m \mu(s)u_1(s)ds=\int_0^m h(s)ds$$
whence
$$\int_0^\infty \mu(s)u_1(s)ds\leq \|h\|_{L^1}<\infty.$$
Consequently
$$\int_{\eta}^\infty u_1(s)ds<\infty$$
and $u_1\in L^1(0,\infty)$. The second equation of \eqref{Eq:Sys_SpecDtilde} implies that
$$(\lambda-A^2_{c_2})u_2=h_2+c_1u_1\in L^1(0,\infty).$$
By Remark \ref{Rem:Estimates}, we have $s(A^2_{c_2})<0$, so $0\in \rho(A^2_{c_2})$ and $u_2\in D(A^2_{c_2})\subset L^1(0,\infty)$. Consequently 
$$U\in D(\B)$$
so $0\in \rho(\B)$
and 
$$s\left(\B\right)<0.$$
\end{proof}

\begin{remark}\label{Rem:Lods}
We suspect that the spectra of $A^1_{\mu}, A^2_{c_2}$ and $\B$ are invariant by translation along the imaginary axis (and therefore are half-spaces), in the spirit of \cite{Lods2009}. We conjecture also that their spectrum consist of essential spectrum only.
\end{remark}

Under suitable assumptions, we can compute $s(\B)$.
\begin{theorem}\label{Thm:SpecBound_Pol}
Suppose that the limits
$$l_{\mu}:=\lim_{x\to \infty}\mu(x), \qquad l_1:=\lim_{x\to \infty}c_1(x)$$
exist and that $c_2\in \R_+$. Then 
$$s(\B)=\dfrac{-(l_1+c_2+l_{\mu})+\sqrt{(l_1+c_2+l_\mu)^2-4l_\mu c_2}}{2}.$$
\end{theorem}

\begin{proof}
If $l_{\mu}=0$, then it is clear, with Theorem \ref{Thm:Spec_Dtilde}, that $s(\B)=0$.
If $c_2=0$, then $s(A^2_{c_2})=0$
by Remark \ref{Rem:Estimates}. Since $B_2$ is a positive operator, we readily see that
\begin{equation}\label{Eq:Bound_B}
s(\B)\geq s(A+B_1)=\max\{s(A^1_{\mu+c_1}),s(A^2_{c_2})\}.
\end{equation}
Consequently $s(\B)\geq 0$ and the equality holds by Theorem \ref{Thm:Spec_Dtilde}. Suppose now that
$$c_2>0, \qquad l_{\mu}>0.$$
Define the second order polynomial function
$$P:\lambda \mapsto \lambda^2 +\lambda (l_1+c_2+l_\mu)+l_\mu c_2$$
whose discriminant is
$$\Delta=l_1^2+2l_1c_2+2l_\mu l_1+(c_2-l_\mu)^2\geq 0$$
and let
$$\lambda^*:=\dfrac{-(l_1+c_2+l_{\mu})+\sqrt{(l_1+c_2+l_\mu)^2-4l_\mu c_2}}{2}<0.$$
We know by Theorem \ref{Thm:Spec_Dtilde} that $$s(\B)<0$$
since $c_2>0$ and $l_{\mu}>0$. Let $\varepsilon\in(0,-\lambda^*)$, $\lambda:=\lambda^*+\varepsilon<0$ and $(h_1,h_2)\in \X_+$. The solution $U:=(u_1,u_2)$ of
$$(\lambda I-\B)U=H, \qquad U(0)=(0,0)$$ satisfies \eqref{Eq:Sys_SpecDtilde}. We multiply the first equation by $(\lambda+c_2)$ and the second one by $c_2$, then we do the sum of both equations. We obtain:
\begin{equation}\label{Eq:Sum_Pol}
(\lambda+c_2)(\gamma_1 u_1)'+c_2(\gamma_2 u_2)'+[\lambda^2+\lambda(c_1+c_2+\mu)+\mu c_2]u_1=(\lambda+c_2)h_1+c_2h_2=:h
\end{equation}
where $h\in L^1(0,\infty)$. By assumptions made on $c_1$ and $\mu$, we know that for every $\eta>0$, there exists $\delta>0$ such that
$$|\mu(s)-l_\mu|\leq \eta, \qquad |c_1(s)-l_1|\leq \eta, \qquad \forall s\geq \delta.$$
Moreover, we have
\begin{eqnarray*}
&&\lambda^2 +\lambda(c_1(s)+c_2(s)+\mu(s))+\mu(s)c_2(s)\\
&\geq&(\lambda^*+\varepsilon)^2+(\lambda^*+\varepsilon)(l_1+c_2+l_\mu+2\eta)+c_2(l_\mu-\eta)\\
&=& \varepsilon^2+2\varepsilon \lambda^*+2\eta \lambda^*+\varepsilon(l_1+c_2+l_\mu+2\eta)-\eta c_2\\
&=& \varepsilon[2\lambda^*+(l_1+c_2+l_\mu)]+\varepsilon^2+2\lambda^* \eta+2\varepsilon \eta -\eta c_2 \\
&\geq& \varepsilon^2+2\lambda^* \eta+2\varepsilon \eta -\eta c_2=:C(\eta)
\end{eqnarray*}
for every $s\geq \delta$, since $P(\lambda^*)=0$ and
$$2\lambda^*\geq -(l_1+c_2+l_\mu).$$
We see that $C(0)=\varepsilon^2>0$. Since $C$ is a continuous function, then we can find $\eta^*>0$ small enough such that $C(\eta^*)>0$. Thus there exists $\delta>0$ such that for every $s\geq \delta$, we have
$$\lambda^2 +\lambda(c_1(s)+c_2(s)+\mu(s))+\mu(s)c_2(s)\geq C(\eta^*)>0.$$
An integration of \eqref{Eq:Sum_Pol} and some lower bounds lead to
$$(\lambda+c_2)\int_{\delta}^m (\gamma_1 u_1)'(s)ds+c_2\int_{\delta}^m (\gamma_2 u_2)'(s)ds+C(\eta^*)\int_{\delta}^m u_1(s)ds\leq \int_{\delta}^m h(s)ds$$
for every $m>\delta$. Consequently
$$C(\eta^*)\int_{\delta}^\infty u_1(s)ds\leq \|h\|_{L^1(0,\infty)}+(\lambda+c_2)\gamma_1(\delta)u_1(\delta)+c_2 \gamma_2(\delta)u_2(\delta)<\infty.$$
Finally $u_1\in L^1(0,\infty)$ and, using the second equation of \eqref{Eq:Sys_SpecDtilde},  we get $u_2\in L^1(0,\infty)$. Consequently  we have
$$\lambda^*+\varepsilon\in \rho(\B)$$
for every $\varepsilon>0$, so
$$s(\B)\leq \lambda^*.$$
If $l_1=0$, then we have
$$\lambda^*=\max\{-l_{\mu},-c_2\}$$
and
$$\max\{s(A^1_{\mu+c_1}), s(A^2_{c_2})\}=\max\{-l_{\mu},-c_2\},$$
by using Theorem \ref{Prop:Spec_Bmu} and Remark \ref{Rem:Estimates}. Consequently, using \eqref{Eq:Bound_B}, we get
$$s(\B)\geq \lambda^*$$
and the equality holds. Suppose in the following that
$$l_1>0.$$
We see that
$$P(-c_2)=-l_1c_2<0$$
so we have 
$$\lambda^*>-c_2.$$
Let $H\in \X_+$, $\lambda:=\lambda^*-\varepsilon<0$, with $\varepsilon>0$ small enough such that $\lambda>-c_2$ (which is possible since $\lambda^*>-c_2$). Suppose that $\lambda\in \rho(\B)$, then $U:=(\lambda I-\B)^{-1}H=(u_1,u_2)$ satisfies \eqref{Eq:Sum_Pol}. By assumptions on the parameters, we have
\begin{eqnarray*}
&&\lambda^2 +\lambda(c_1(s)+c_2(s)+\mu(s))+\mu(s)c_2(s)\\
&\leq&(\lambda^*-\varepsilon)^2+(\lambda^*-\varepsilon)(l_1+c_2+l_\mu-2\eta)+c_2(l_\mu+\eta)\\
&=& \varepsilon^2-2\varepsilon \lambda^*-2\eta \lambda^*-\varepsilon(l_1+c_2+l_\mu-2\eta)+\eta c_2\\
&=& \varepsilon^2-\varepsilon[2\lambda^*+(l_1+c_2+l_\mu)]-2\lambda^* \eta+2\varepsilon \eta +\eta c_2\\
&\leq& \varepsilon^2-\varepsilon (l_1-2\eta)-2\lambda^* \eta+\eta c_2:=\tilde{C}(\eta)
\end{eqnarray*}
for every $s\geq \delta$, since
$$2\lambda^*+(l_1+c_2+l_\mu)=\sqrt{\Delta}\geq l_1.$$
Taking $\varepsilon$ small enough such that $\varepsilon\leq l_1/2$, lead to 
$$\tilde{C}(0)=\varepsilon(\varepsilon-l_1)<0.$$
By continuity of $\tilde{C}$, we can find $\eta^*$ small enough such that $\tilde{C}(\eta^*)<0$. Thus there exists $\delta>0$ such that
$$\lambda^2 +\lambda(c_1(s)+c_2(s)+\mu(s))+\mu(s)c_2(s)\leq \tilde{C}(\eta^*)<0, \quad \forall s\geq \delta.$$
An integration of \eqref{Eq:Sum_Pol} between $\delta$ and $\infty$ leads to
$$0\geq -(\lambda+c_2)\gamma_1(\delta)u_1(\delta)-c_2 \gamma_2(\delta)u_2(\delta)+\tilde{C}(\eta^*)\int_{\delta}^\infty u_1(s)ds\geq \int_{\delta}^\infty h(y)dy.$$
We choose $(h_1,h_2)\in \X_+$ such that $\int_{\delta}^\infty h(y)dy>0$ to get a contradiction. We obtain
$$\lambda^*-\varepsilon\in \sigma(\B)$$
for every $\varepsilon>0$ small enough, whence
$$s(\B)\geq \lambda^*$$
and the equality follows.
\end{proof}

\subsection{On the existence of the spectral gap}\label{Sec:SpectralGap}

This subsection deals with different cases where one can check directly the existence or not of a spectral gap.

\subsubsection{Sub (resp. super) conservative systems}

We start with:

\begin{theorem}\label{Thm:Conserv_Gap}
Suppose that 
$$\int_0^\infty \beta(s,y)ds\geq \mu(y), \quad \forall y\geq 0$$
and
$$\liminf_{x\to \infty}\mu(x)>0, \qquad \liminf_{x\to \infty}c_2(x)>0.$$
Then we have $s(\A)\geq 0$ and $s(\B)<0$. In particular $\{T_{\A}(t)\}_{t\geq 0}$ has a spectral gap, i.e. $\omega_{\ess}(\A)<\omega_0(\A)$.
\end{theorem}

\begin{proof}
The fact that $s(\B)<0$ is given by Theorem \ref{Thm:Spec_Dtilde}. To prove that $s(\A)\geq 0$, let the initial condition $(u^0_1,u^0_2)\in D(\A)\cap \X_+.$ An integration of \eqref{Eq:Model_inf} gives us
\begin{eqnarray*}
\dfrac{d}{dt}\left[\int_0^\infty u_1(s,t)ds\right]&=&-\int_0^\infty (\mu(s)+c_1(s))u_1(s,t)ds+\int_0^\infty c_2(s)u_2(s,t)ds\\
&&+\int_0^\infty \int_0^\infty \beta(s,y)u_1(y,t)dyds\\
\dfrac{d}{dt}\left[\int_0^\infty u_2(s,t)ds\right]&=&-\int_0^\infty c_2(s))u_2(s,t)ds+\int_0^\infty c_1(s)u_1(s,t)ds
\end{eqnarray*}
for every $t\geq 0$. The sum of the latter equations then lead to
\begin{eqnarray*}
&&\dfrac{d}{dt}\left[\int_0^\infty (u_1(s,t)+u_2(s,t))ds\right] \\
&=&\int_0^\infty \int_0^\infty \beta(s,y)u_1(y,t)dyds-\int_0^\infty \mu(s)u_1(s,t)ds\\
&=& \int_0^\infty \left[\int_0^\infty \beta(s,y)ds-\mu(y)\right]u_1(y,t)dy\geq 0
\end{eqnarray*}
by assumption. Consequently we get
$$\left\|T_{\A}(t)\begin{pmatrix}
u^0_1 \\
u^0_2
\end{pmatrix}\right\|\geq \left\|\begin{pmatrix}
u^0_1 \\
u^0_2
\end{pmatrix}\right\| \quad \forall t\geq 0.$$
By density of $D(\A)\cap \X_+$ in $\X_+$, the latter inequality holds for every $(u^0_1,u^0_2)\in \X_+$ and
$$\|T_{\A}(t)\|_{\L(\X)}\geq 1$$
for every $t\geq 0$. Consequently we have
$$\omega_0(\A)\geq 0$$
and
$$s(\A)\geq 0.$$
\end{proof}
We give now a `converse' result
\begin{theorem}\label{Thm:NoSpecGap}
Suppose that 
\begin{equation}\label{Eq:Hyp_C2-mu}
\lim_{x\to \infty}c_2(x)=0 \text{ or } \lim_{x\to \infty}\mu(x)=0
\end{equation}
and that
\begin{equation}\label{Eq:HypBeta_Inf}
\int_0^\infty \beta(s,y)ds\leq \mu(y), \quad \forall y\geq 0.
\end{equation}
Then $s(\B)=0$ and $s(\A)=0$. In particular $\{T_{\A}(t)\}_{t\geq 0}$ has not a spectral gap, i.e. $\omega_{\ess}(\A)=\omega_0(\A)$.
\end{theorem}

\begin{proof}
If $\lim_{x\to \infty}\mu(x)=0$, then it is clear that 
$$s(\B)=0$$
by Theorem \ref{Thm:Spec_Dtilde}. If $\lim_{x\to \infty}c_2(x)=0$, then, using Remark \ref{Rem:Estimates}, we see that $s(A^2_{c_2})=0$. The fact that $s(\B)=0$ follows from Theorem \ref{Thm:Spec_Dtilde} and \eqref{Eq:Bound_B}. Let the initial condition $(u^0_1,u^0_2)\in D(\A)\cap \X_+$. An integration of \eqref{Eq:Model_inf} gives us
\begin{eqnarray*}
\dfrac{d}{dt}\left[\int_0^\infty (u_1(s,t)+u_2(s,t))ds\right]= \int_0^\infty \left[\int_0^\infty \beta(s,y)ds-\mu(y)\right]u_1(y,t)dy\leq 0.
\end{eqnarray*}
By density, we then get
$$\left\|T_{\A}(t)\begin{pmatrix}
u^0_1 \\
u^0_2
\end{pmatrix}\right\|\leq \left\|\begin{pmatrix}
u^0_1 \\
u^0_2
\end{pmatrix}\right\| \quad \forall t\geq 0,$$
for every $(u^0_1, u^0_2)\in \X_+$. Consequently, we have
$$\|T_{\A}(t)\|_{\L(\X)}\leq 1$$
for every $t\geq 0$ so $\omega_0(\A)\leq 0$ and
$$s(\A)\leq 0.$$
Since $\A$ is a positive and bounded perturbation of $\B$, we have
$$s(\A)\geq s(\B).$$
It then follows from Remark \ref{Rem:ess-type}, that
$$\omega_{\ess}(\A)=\omega_0(\A)$$
which ends the proof.
\end{proof}

We note that in contrast to the case $m<\infty$, the irreducibility of the semigroup does not imply the existence of spectral gap since \eqref{Eq:Hyp_C2-mu} and \eqref{Eq:HypBeta_Inf} are compatible with the irreducibility of the semigroup.

\subsubsection{A particular case}

We show now that the spectral gap is always present when some parameters are constant.

\begin{theorem}\label{Thm:Partic_Case}
Let $c_1, c_2$ and $\mu$ be positive constants. If $\beta_1(s):=\inf_{y\geq 0} \beta(s,y)$ is not identically zero then
$$s(\A)>s(\B).$$
In particular $\{T_{\A}(t)\}_{t\geq 0}$ has a spectral gap.
\end{theorem}

\begin{proof}
The computation of $s(\B)$ follows from Theorem \ref{Thm:SpecBound_Pol}:
$$s(\B)=\dfrac{-(c_1+c_2+\mu)+\sqrt{(c_1+c_2+\mu)^2-4\mu c_2}}{2}=:\lambda^*$$
Let
$$\lambda:=\lambda^*+\varepsilon \quad (\varepsilon>0).$$
If $\lambda>s(\A)$ then $\lambda\in \rho(\A)$ and 
$(\lambda-\A)^{-1}$ is positive. So for any $(h_1,h_2)\in \X_+-\left\{0\right\}$,
$$(u_1,u_2):=(\lambda-\A)^{-1}(h_1,h_2)$$
is nonnegative and satisfies
\begin{equation*}
\left\{
\begin{array}{rcl}
(\gamma_1 u_1)'+(\lambda+c_1+\mu)u_1-c_2 u_2-\int_0^\infty \beta(\cdot,y)u_1(y)dy&=&h_1, \\
(\gamma_2 u_2)'+(\lambda+c_2)u_2-c_1u_1&=&h_2.
\end{array}
\right.
\end{equation*}
We multiply the first equation by $\lambda+c_2$ and the second one by $c_2$, then the sum implies that
\begin{eqnarray*}
&&(\lambda+c_2)(\gamma_1 u_1)'+c_2(\gamma_2 u_2)'+[\lambda^2+\lambda(c_1+c_2+\mu)+\mu c_2]u_1 \\
&=&(\lambda+c_2)\int_0^\infty \beta(\cdot,y)u_1(y)dy+h,
\end{eqnarray*}
where $h:=(\lambda+c_2)h_1+c_2h_2$. An integration of the latter equation leads to
\begin{eqnarray*}
&&[\lambda^2+\lambda(c_1+c_2+\mu)+\mu c_2]\int_0^\infty u_1(y)dy \\
&=&\int_0^\infty h(y)dy+(\lambda+c_2)\int_0^\infty \int_0^\infty \beta(s,y)u_1(y)dyds
\end{eqnarray*}
and replacing $\lambda$ by its expression, we obtain
\begin{eqnarray*}
&&[\varepsilon^2+\varepsilon(2\lambda^*+c_1+c_2+\mu)]\int_0^\infty u_1(y)dy\\
&=&\int_0^\infty h(y)dy+(\lambda^*+\varepsilon+c_2)\int_0^\infty \int_0^\infty \beta(s,y)u_1(y)dyds.
\end{eqnarray*}
Consequently, we have
\begin{equation}\label{Eq:Const_Sum}
f(\varepsilon)\int_0^\infty u_1(y)dy\geq \int_0^\infty h(y)dy,
\end{equation}
where we defined
$$f:\varepsilon \mapsto [\varepsilon^2+\varepsilon(2\lambda^*+c_1+c_2+\mu)]-(\lambda^*+\varepsilon+c_2)\int_0^\infty \beta_1(s)ds.$$
Since $\lambda^*>-c_2$, then
$$f(0)=-(\lambda^*+c_2)\int_0^\infty \beta_1(s)ds<0.$$
The fact that $\lim_{\varepsilon \to \infty}f(\varepsilon)=\infty$ implies, by continuity, that there exists $\overline{\varepsilon}>0$ such that $f(\overline{\varepsilon})=0$. Considering $\varepsilon\in[0,\overline{\varepsilon}]$ in \eqref{Eq:Const_Sum} would lead to
$$0\geq f(\varepsilon)\int_0^\infty u_1(y)dy\geq \int_0^\infty h(y)dy>0$$
which is a contradiction. Hence $\lambda^*+\overline{\varepsilon}\leq s(\A)$ and this ends the proof.
\end{proof}

\begin{remark}\label{Rem:Bound_SpecGap}
A simple computation shows that
$$\overline{\varepsilon}=\dfrac{-(2\lambda^*+c_1+c_2+\mu-\int_0^\infty \beta_1(s)ds)+\sqrt{\Delta}}{2}>0$$
where
$$\Delta:=\left(2\lambda^*+c_1+c_2+\mu-\int_0^\infty \beta_1(s)ds\right)^2+4(\lambda^*+c_2)\int_0^\infty \beta_1(s)ds>0$$
which provides us with an explicit lower bound of the spectral gap 
$$s(\A)-s(\B)\geq \overline{\varepsilon}.$$
\end{remark}

% You may incorporate your references as follows in your main tex file.
% Using BibTex is not recommended but can be handled.

\medskip
% The data information below will be filled by AIMS editorial staff
Received xxxx 20xx; revised xxxx 20xx.
\medskip

\end{document}